\documentclass[preprint]{elsarticle}
\usepackage{amsthm,amsmath,amssymb}
\usepackage{color,amssymb}
\usepackage{enumerate,mathdots,multicol}

\newtheorem{Theorem}{Theorem}[section]
\newtheorem{Lemma}[Theorem]{Lemma}
\newtheorem{Proposition}[Theorem]{Proposition}

\newenvironment{Remark}{{\bf Remark.}}

\newcommand{\A}{{\mathcal{A}}}

\newcommand{\C}{{\mathbb C}}

\newcommand{\Ff}{{\mathbb F}}

\newcommand{\RR}{\mathbb{R}}

\newcommand{\mm}{\mathcal{M}}

\newcommand{\wt}{\widetilde{\Phi}}

\newcommand{\s}{\mathcal{H}}

\DeclareMathOperator{\rank}{rank}
\DeclareMathOperator{\Syl}{Syl}
\DeclareMathOperator{\sign}{sign}
\DeclareMathOperator{\id}{id}

\journal{arXiv.org}

\begin{document}

\begin{frontmatter}
\title{Jordan triple product homomorphisms on Hermitian matrices of dimension two}

\author{Damjana Kokol Bukov\v sek}
\ead{damjana.kokol.bukovsek@ef.uni-lj.si}
\address{Faculty of Economics, University of Ljubljana, Kardeljeva plo\v s\v cad 17, Ljubljana,
and Institute of Mathematics, Physics and Mechanics, Department of Mathematics, Jadranska~19, Ljubljana, Slovenia.} 

\author{Bla\v z Moj\v skerc}
\ead{blaz.mojskerc@ef.uni-lj.si}
\address{Faculty of Economics, University of Ljubljana, Kardeljeva plo\v s\v cad 17, Ljubljana,
and Institute of Mathematics, Physics and Mechanics, Department of Mathematics, Jadranska~19, Ljubljana, Slovenia.} 

\begin{abstract}
We characterise all Jordan triple product homomorphisms, that is, mappings $\Phi$ satisfying
$$ \Phi(ABA) = \Phi(A)\Phi(B)\Phi(A) $$
on the set of all Hermitian $2 \times 2$ complex matrices.
\end{abstract}

\end{frontmatter}

\section{Introduction}

In order to understand the geometry of matrix spaces, mappings with certain properties are often studied. Among such properties is (anti)multi\-plica\-tivity. The structure of (anti)multiplicative mappings on the algebra $\mm_n(\Ff)$ of $n\times n$ matrices over field $\Ff$ is well understood \cite{JL}, but less is known about (anti)multiplicative mappings from $\mm_n(\Ff)$ to $\mm_m(\Ff)$ for $m>n$.

In a well known survey paper \cite{Semrl} \v{S}emrl presented many facts and properties of such mappings, along with properties of preservers of Jordan and Lie product. \v{S}emrl exposed a related problem, that is, to characterize maps that are multiplicative with respect to {\it Jordan triple product} (J.T.P. for short), namely maps $\Phi$ on $\mm_n(\Ff)$ satisfying 
$$ \Phi(ABA) = \Phi(A)\Phi(B)\Phi(A) $$
for all $A,B \in \mm_n(\Ff)$. Such mappings were studied under additional assumption of additivity on quite general
domain of certain rings \cite{Bresar}. In response to \v{S}emrl, Kuzma characterized nondegenerate J.T.P. homomorphisms on the set $\mm_n(\Ff)$ in \cite{Kuzma}
for $n \ge 3$, in \cite{Dob1} Dobovi\v sek characterized J.T.P. homomorphisms from $\mm_n(\Ff)$ to $\Ff$, and in \cite{Dob2} he
characterized J.T.P. homomorphisms from $\mm_2(\Ff)$ to $\mm_3(\Ff)$.

\medskip

In this paper we focus on J.T.P. homomorphisms on the set of all Hermitian complex $2 \times 2$ matrices. By $A^*$ denote the
complex conjugate of the transpose of matrix $A$ and by $\s_2(\C)$ the set of all Hermitian complex $2 \times 2$ matrices
$$ \s_2(\C) = \{ A \in \mm_2(\C); A = A^*\}. $$
We cannot study multiplicative or antimultiplicative maps on Hermitian matrices, since they are not closed under multiplication. 
But they are closed under J.T.P., so studying J.T.P. homomorphisms on Hermitian matrices makes perfect sense. Characterization 
of J.T.P. homomorphisms on the set of Hermitian matrices may shed a new light on the structure of Hermitian matrices and may be useful 
in the areas where only Hermitian or positive (semi)definite matrices appear, such as some areas of financial mathematics.

Jordan triple product homomorphisms were already studied on the set of positive definite matrices, Gselmann \cite{Gsel} characterized mappings from the set of positive definite real or complex matrices to the field of real numbers.
In the paper \cite{KBM}, similar result was proved, namely Jordan triple product homomorphisms 
from the set of all Hermitian $n \times n$ complex matrices to the field of complex numbers and
Jordan triple product homomorphisms from the field of complex or real numbers
or the set of all nonnegative real numbers to the set of all Hermitian $n \times n$ complex matrices were characterized.
Further, Hao et al. \cite{Hao} characterized injective Jordan triple product endomorphisms on the set of complex symmetric matrices, and Molnar in \cite{Molnar} described continuous Jordan triple endomorphisms on the set of complex positive definite matrices of size at least 3. 
The special case of $2 \times 2$ positive definite complex matrices was considered separately in \cite{MV}. One may think that in this case the solution can be found straightforwardly, but this is far from being true. We generalize this result by omitting the continuity assumption and enlarging the set of matrices to all complex Hermitian matrices.

\medskip

The paper is organized as follows. In section 2 we state the characterization theorem for J.T.P. homomorphisms on $\s_2(\C)$. In section 3 we list some results on J.T.P. homomorphisms on the set $\s_n(\C)$ and main results from \cite{KBM} which we will find useful later on. In sections 4--7 we treat different cases of J.T.P. homomorphisms, namely irregular, scalar, nondegenerate and degenerate cases. 

\section{Characterization Theorem}

We first introduce some notation.
By $I$ we denote the identity matrix of an appropriate dimension, by $\det A$ the determinant and by $\rank A$ the 
rank of a matrix $A$. By $\sigma(A)$ we denote the spectrum of a matrix $A$, and by $\Syl(A)$ the inertia of $A$, that is, the number of positive eigenvalues of $A$.
The direct sum $A \oplus B$ is a block diagonal matrix $\left[
\begin{array}{cc}
A & 0 \\
0 & B
\end{array} \right]$. 
The notation $A>0$ means that a matrix $A\in \s_2(\C)$ is positive definite, $A<0$ is a negative definite matrix and $A<>0$ is an invertible nondefinite matrix.

\medskip

We can now state our main result.

\begin{Theorem} \label{main}
Let $\Phi: \s_2(\C) \to \s_2(\C)$. Then $\Phi(ABA)=\Phi(A)\Phi(B)\Phi(A)$ if and only if there exists some unitary matrix $U\in\mm_2(\C)$ such that $\Phi$ has one of the following forms:
\begin{enumerate}[(i)]
	\item $\Phi(A)=U \begin{bmatrix} \varphi_1(A) & 0 \\ 0 & \varphi_2(A) \end{bmatrix} U^*$ where $\varphi_1,\varphi_2: \s_2(\C) \to \RR$ are J.T.P. homomorphisms having the form $\varphi_i(A)=\psi_i(|\det A|) \eta_i(\Syl(A))$ for $i=1, 2,$ 
with $\psi_1, \psi_2: [0,\infty) \to \C$ multiplicative functions, $\eta_1, \eta_2:\{0,1,2\} \to \{-1,1\}$ arbitrary mappings, and $\Syl(A)$ the inertia of $A$;
	\item $\Phi(A)=\pm UAU^*$;
	\item $\Phi(A)= \pm U\bar{A}U^*= \pm UA^TU^*$;
	\item $\Phi(A)= \begin{cases} \pm \beta(\det A)\cdot U \widetilde{\Phi}(A) U^*; & \rank A=2 \\ 0 & \rank A \leq 1 \end{cases}$, where $\beta:\RR^* \to \RR^*$ is a unital multiplicative map, and $\widetilde{\Phi}$ has one of the following forms:
		\begin{multicols}{2}
		\begin{itemize}
			\item $\wt(A)=A$;
			\item $\wt(A)=\bar{A}$;
			\item $\wt(A)=A^{-1}$;
			\item $\wt(A)=\bar{A}^{-1}$;
			\item $\wt(A)=\eta(A)A$;
			\item $\wt(A)=\eta(A) \bar{A}$;
			\item $\wt(A)=\eta(A) A^{-1}$;
			\item $\wt(A)=\eta(A) \bar{A}^{-1}$;
		\end{itemize}
		\end{multicols}
\end{enumerate}
with $$\eta(A)= \left\{
\begin{array}{rl}
1; & A>0 \textrm{ or } A<>0 \\
-1; & A<0
\end{array} \right. .$$
\end{Theorem}

\noindent
It is obvious that mappings of the forms described in (i)--(iv) are J.T.P. homomorphisms on $\s_2(\C)$.

\section{Preliminaries}

In this section we present some properties of J.T.P. homomorphisms on the set $\s_n(\C)$ we will use later on. These properties with proofs can be found in \cite{KBM}. 
We start with a simple lemma.

\begin{Lemma}[Lemma 2.1 in \cite{KBM}] \label{lemma21}
Let $A\in \s_2(\C)$ be a Hermitian matrix. Then there exists a unitary Hermitian matrix $B\in \s_2(\C)$ such that
$A=B (\lambda_1 \oplus \lambda_2) B$
with $\lambda_1, \lambda_2 \in \RR$.
\end{Lemma}


We continue with characterization of J.T.P. homomorphisms mapping from $\s_2(\C)$ to $\C$.
 
\begin{Lemma}[Lemma 3.1 in \cite{KBM}] \label{rank-less-n}
Let $\Phi: \s_2(\C) \to \C$ be a J.T.P. homomorphism with $\Phi(I)=1$ and $\Phi(0)=0$. Then $\Phi(A)=0$ for every $A\in \s_2(\C)$ with $\rank A<2$.
\end{Lemma}

\begin{Proposition}[Theorem 3.3 in \cite{KBM}] \label{thm3.3}
Let $\Phi$ be a mapping from $\s_2(\C)$ to $\C$. Then $\Phi$ satisfies the identity $\Phi(ABA)=\Phi(A) \Phi(B) \Phi(A)$ 
if and only if $\Phi$ has the form
$$\Phi(A)=\Psi(|\det A|) \eta(\Syl(A)),$$ 
where $\Psi: [0,\infty) \to \C$ is a multiplicative function, $\eta:\{0,1,2\} \to \{-1,1\}$ an arbitrary mapping, and $\Syl(A)$ the inertia of $A$.
\end{Proposition}

We also need the characterization of J.T.P. homomorphisms from matrices of dimension one to $n \times n$ Hermitian matrices. 

\begin{Lemma}[Lemma 4.1 in \cite{KBM}] \label{C_to_invertible}
Let a mapping $\Phi: \A \to \s_n(\C)$ be a J.T.P. homomorphism where $\A$ is the set $\C^*$, $\RR^*$, or $\RR^+$, such that 
$\Phi(\lambda)$ is invertible for every $\lambda \in \A$ and $\Phi(1) = I$. Then there exist a unitary matrix $U$ and 
multiplicative maps $\varphi_1, \varphi_2: \A \to \RR^*$ with $\varphi_i(1) = 1$, such that
$$\Phi(\lambda) = U (\varphi_1(\lambda) \oplus \varphi_2(\lambda)) U^*, \quad \lambda \in \A.$$
\end{Lemma}

\begin{Proposition}[Theorem 4.2 in \cite{KBM}] \label{C_to_all}
Let a mapping $\Phi: \A \to \s_n(\C)$ be a J.T.P. homomorphism where $\A$ is the set $\C$, $\RR$, or $\RR^+\cup\{0\}$.
Then there exist a unitary matrix $U$, a diagonal matrix $D$ with $\pm 1$'s on its diagonal and 
multiplicative maps $\varphi_1, \varphi_2: \A \to \RR$, such that
$$\Phi(\lambda) = U D(\varphi_1(\lambda) \oplus \varphi_2(\lambda)) U^*, \ \ \lambda \in \A.$$
\end{Proposition}

We finish this section with a characterization of $2 \times 2$ Hermitian involutions $A$, that is, $A^2=I$.

\begin{Lemma} \label{involution}
Let $A$ be a Hermitian matrix. Then $A$ is an involution if and only if
$A=\pm I$ or
$$A=\begin{bmatrix} \pm \sqrt{1-|a|^2} & a \\ \bar{a} & \mp \sqrt{1-|a|^2} \end{bmatrix}$$
for some $a\in \C$ with $|a| \leq 1$.
\end{Lemma}

\begin{proof}
Use a direct calculation.
\end{proof}

\section{Irregular cases}

In this section we start with the study J.T.P. homomorphisms that map from $2\times 2$ Hermitian matrices to $2\times 2$ Hermitian matrices. Since $\Phi(0)= \Phi(0^3)=\Phi(0)^3$, it must be that $\sigma(\Phi(0)) \subset \{-1,0,1\}$. So we consider several cases.

\medskip

\noindent
{\sc Case 1:} If $\Phi(0)$ is invertible, then it follows from
$$\Phi(0)=\Phi(0\cdot A \cdot 0)=\Phi(0) \Phi(A) \Phi(0)$$
that $\Phi(A)=\Phi(0)^{-1}=\Phi(0)$ for every $A \in \s_2(\C)$ with $\Phi(0)$ some involution in $\s_2(\C)$.

\medskip

\noindent
{\sc Case 2:} If $\rank \Phi(0)=1$, then it follows from $\Phi(0)=\Phi(0)^3$ that $\sigma(\Phi(0))=\{0,\alpha\}$ with $\alpha \in \{-1,1\}$. Hence we can write
$$\Phi(0)=U
	\begin{bmatrix}
		\alpha & 0 \\
		0 & 0 
	\end{bmatrix} U^*$$
for some unitary matrix $U\in \mm_2(\C)$. Choose an arbitrary $A \in \s_2(\C)$ and write
$$\Phi(A)= U
	\begin{bmatrix}
		a & b \\
		\bar{b} & c
	\end{bmatrix} U^*.$$
Then
\begin{align*}
	\Phi(0) &= \Phi(0) \Phi(A) \Phi(0) 
	= U \begin{bmatrix} \alpha & 0 \\ 0 & 0 \end{bmatrix} \begin{bmatrix} a & b \\ \bar{b} & c \end{bmatrix} \begin{bmatrix} \alpha & 0 \\ 0 & 0 \end{bmatrix} U^* 
	= U \begin{bmatrix} a & 0 \\ 0 & 0 \end{bmatrix} U^*.
\end{align*}
Hence $a=\alpha$.
On the other hand, 
\begin{align*}
\Phi(0)=\Phi(A) \Phi(0)\Phi(A)= U \begin{bmatrix} a & b \\ \bar{b} &c \end{bmatrix} \begin{bmatrix} \alpha & 0 \\ 0 & 0 \end{bmatrix} \begin{bmatrix} a & b \\ \bar{b} & c \end{bmatrix} U^* =
U \begin{bmatrix} \alpha & b \\ \bar{b} & \alpha |b|^2 \end{bmatrix} U^*,
\end{align*}
from which it follows $b=0$. We conclude that for every $A\in \s_2(\C)$
$$\Phi(A)= U \begin{bmatrix} \alpha & 0 \\ 0 & \varphi(A) \end{bmatrix} U^*$$
for some J.T.P. homomorphism $\varphi: \s_2(\C) \to \RR$ with $\varphi(0)=0$.

\medskip

We split the remaining case $\Phi(0)=0$ into several subcases, depending on the image $\Phi(I)$. Since $\Phi(I)=\Phi(I)^3$, it must be that $\sigma(\Phi(I)) \subset \{-1,0,1\}$.

\medskip

\noindent
{\sc Case 3:} Let $\Phi(I)=0$. Then $\Phi(A)=\Phi(I)\Phi(A)\Phi(I)=0$ for every $A\in \s_2(\C)$.

\medskip

\noindent
{\sc Case 4:} If $\rank \Phi(I)=1$, we write
$$\Phi(I)= U \begin{bmatrix} \alpha & 0 \\ 0 & 0 \end{bmatrix} U^*$$
for $\alpha \in \{-1,1\}$ and a unitary matrix $U\in \mm_2(\C)$. Write
$$\Phi(A)= U \begin{bmatrix} a & b \\ \bar{b} & c \end{bmatrix} U^*.$$
Then for every $A\in \s_2(\C)$ we get
\begin{align*}
\Phi(A) &= \Phi(I)\Phi(A) \Phi(I)= U \begin{bmatrix} \alpha & 0 \\ 0 & 0 \end{bmatrix} \begin{bmatrix} a & b \\ \bar{b} & c \end{bmatrix} \begin{bmatrix} \alpha & 0 \\ 0 & 0 \end{bmatrix} U^* \\
 &= U \begin{bmatrix} a & 0 \\ 0 & 0\end{bmatrix} U^* = U\begin{bmatrix} \varphi(A) & 0 \\ 0 & 0 \end{bmatrix} U^*
\end{align*}
for some J.T.P. homomorphism $\varphi: \s_2(\C) \to \RR$ with $\varphi(0)=0$ and $\varphi(I)=\alpha$.

\medskip

\noindent
{\sc Case 5:} Let $\Phi(I)$ be invertible. From $\Phi(I)=\Phi(I)^3$ it follows that $\Phi(I)^2=I$. Denote $P:=\Phi(I)$. Then $$\Phi(A)=\Phi(I)\Phi(A)\Phi(I)=P \Phi(A) P$$ for every $A\in\s_2(\C)$, hence $\Phi(A)P=P\Phi(A)$ for every $A\in\s_2(\C)$. 
If $P \ne \pm I$, we can write $P = U \begin{bmatrix} 1 & 0 \\ 0 & -1\end{bmatrix} U^*$ for some unitary $U \in \mm_2(\C)$. Since $\Phi(A)$ commutes with $P$, we have
$$\Phi(A) = U \begin{bmatrix} \varphi_1(A) & 0 \\ 0 & \varphi_2(A)\end{bmatrix} U^* $$
for some J.T.P. homomorpisms $\varphi_1, \varphi_2: \s_2(\C) \to \RR$.
If $P = -I$, define a mapping $\Phi'(A)=- \Phi(A)$. $\Phi'$ is a J.T.P. homomorphism from $\s_2(\C)$ to $\s_2(\C)$ with $\Phi'(0)=0$ and $\Phi'(I)=I$. This translates directly to the last case in need of considering, Case 6.

\medskip

\noindent
{\sc Case 6:} $\Phi:\s_2(\C) \to \s_2(\C)$ is a J.T.P. homomorphism with $\Phi(0)=0$ and $\Phi(I)=I$. 
We refer to this case as a {\it regular} case.

\medskip

Cases 1--6 amount to the following proposition.

\begin{Proposition} \label{thm3.1}
Let $\Phi: \s_2(\C) \to \s_2(\C)$ be a J.T.P. homomorphism. Then $\Phi$ is regular, or $-\Phi$ is regular, 
or there exist a unitary matrix $U$, such that
$$\Phi(A)=U \begin{bmatrix} \varphi_1(A) & 0 \\ 0 & \varphi_2(A)\end{bmatrix} U^*$$
where $\varphi_1, \varphi_2: \s_2(\C) \to \RR$ are J.T.P. homomorpisms characterised in Proposition \ref{thm3.3},
possibly constant mappings $\varphi_i(A) = c \in \{-1, 0, 1\}$ for all $A \in \s_2(\C)$.
\end{Proposition}


\medskip

All cases when $\Phi(I) \ne \pm I$ or $\Phi(0) \ne 0$ are covered by the form (i) of Theorem \ref{main}. In the case when 
$-\Phi$ is regular, we get the negative sign in the forms (ii) - (iv) of Theorem \ref{main}.

\section{Nontrivial involution to a scalar}

In sections 5--7 we assume  $\Phi:\s_2(\C) \to \s_2(\C)$  to be a regular J.T.P. homomorphism, that is, $\Phi(0)=0$ and $\Phi(I)=I$. We now consider the image of a nontrivial involution
$ J = \begin{bmatrix} 1 & 0 \\ 0 & -1\end{bmatrix} .$
Since $J^2 = I$, it is mapped to an involution. So, $\Phi(J)$ is a matrix similar to $J$, or a scalar matrix $I$ or $-I$. 
In this section we assume that $\Phi(J) \in \{ -I,I \}$.

\begin{Lemma} \label{podobne_v_enake}
Let $\Phi:\s_2(\C) \to \s_2(\C)$ be a regular J.T.P. homomorphism and $\Phi(J) \in \{ -I,I \}.$ 
Then every nontrivial involution is mapped to $\pm I$. If matrices $A, B \in \s_2(\C)$ are similar, then $\Phi(A) = \Phi(B)$.
\end{Lemma}

\begin{proof}
Let $A$ be a nontrivial involution. By Lemma \ref{lemma21} it can be written as $A = BJB$ with
$B^2 = I$. Thus
$\Phi(A) = \Phi(B)\Phi(J)\Phi(B) = \pm \Phi(B)^2 = \pm I.$
If matrices $A, B \in \s_2(\C)$ are similar, then again by Lemma \ref{lemma21} there exist involutions $C$, $E$ and diagonal matrix $D$ such that $A = CDC$ and $B = EDE$. Now $\Phi(A) = \Phi(D) = \Phi(B)$.
\end{proof}

\begin{Lemma} \label{rang_1_v_0}
If $\Phi:\s_2(\C) \to \s_2(\C)$ is a regular J.T.P. homomorphism and $\Phi(J) \in \{ -I,I \}$, then 
$\Phi(A)=0$ for every matrix $A\in \s_2(\C)$ with $\rank A=1$.
\end{Lemma}

\begin{proof}
First notice that 
$$ \Phi\left(\begin{bmatrix} 1 & 0 \\ 0 & 0\end{bmatrix}\right) = \Phi\left(\begin{bmatrix} 0 & 1 \\ 1 & 0\end{bmatrix}
\begin{bmatrix} 0 & 0 \\ 0 & 1\end{bmatrix}\begin{bmatrix} 0 & 1 \\ 1 & 0\end{bmatrix}\right) =
\Phi\left(\begin{bmatrix} 0 & 0 \\ 0 & 1\end{bmatrix}\right) $$
since matrix $\begin{bmatrix} 0 & 1 \\ 1 & 0\end{bmatrix}$ is an involution.
Now
\begin{align*}
\Phi\left(\begin{bmatrix} a & 0 \\ 0 & 0\end{bmatrix}\right) &= \Phi\left(\begin{bmatrix} 1 & 0 \\ 0 & 0\end{bmatrix}
\begin{bmatrix} a & 0 \\ 0 & c\end{bmatrix}\begin{bmatrix} 1 & 0 \\ 0 & 0\end{bmatrix}\right) \\ &= 
\Phi\left(\begin{bmatrix} 0 & 0 \\ 0 & 1\end{bmatrix}
\begin{bmatrix} a & 0 \\ 0 & c\end{bmatrix}\begin{bmatrix} 0 & 0 \\ 0 & 1\end{bmatrix}\right) =
\Phi\left(\begin{bmatrix} 0 & 0 \\ 0 & c\end{bmatrix}\right)
\end{align*}
for any $a, c \in \RR$. Taking $c= 0$, we obtain that $\Phi\left(\begin{bmatrix} a & 0 \\ 0 & 0\end{bmatrix}\right) = 0$.
Since every matrix $A\in \s_2(\C)$ with $\rank A=1$ is similar to some matrix $\begin{bmatrix} a & 0 \\ 0 & 0\end{bmatrix}$, we get
$\Phi(A)=0$.
\end{proof}

\begin{Lemma} 
Let $\Phi:\s_2(\C) \to \s_2(\C)$ be a regular J.T.P. homomorphism with $\Phi(J) \in \{ -I,I \}$. Let $A\in \s_2(\C)$ be invertible. If $A$ is positive definite or $A$ is nondefinite, then $$\Phi(A)=
\Phi\left(\begin{bmatrix} \det A & 0 \\ 0 & 1\end{bmatrix}\right).$$ If $A$ is negative definite, then $$\Phi(A)=
\Phi(-I)\Phi\left(\begin{bmatrix} \det A & 0 \\ 0 & 1\end{bmatrix}\right).$$
\end{Lemma}

\begin{proof}
First notice that 
$$\Phi\left(\begin{bmatrix} a & 0 \\ 0 & 1\end{bmatrix}\right) = \Phi\left(\begin{bmatrix} 0 & 1 \\ 1 & 0\end{bmatrix}
\begin{bmatrix} 1 & 0 \\ 0 & a\end{bmatrix}\begin{bmatrix} 0 & 1 \\ 1 & 0\end{bmatrix}\right) =
\Phi\left(\begin{bmatrix} 1 & 0 \\ 0 & a\end{bmatrix}\right) .$$
Now, if $A$ is positive definite, it is similar to some matrix $\begin{bmatrix} a & 0 \\ 0 & b\end{bmatrix}$ with $a, b >0$.
If $A$ is nondefinite, it is similar to some matrix $\begin{bmatrix} a & 0 \\ 0 & b\end{bmatrix}$ with $a>0$ and $b<0$.
In both cases we get
\begin{align*}
\Phi(A) &= \Phi\left(\begin{bmatrix} a & 0 \\ 0 & b\end{bmatrix}\right) = \Phi\left(\begin{bmatrix} \sqrt{a} & 0 \\ 0 & 1\end{bmatrix}
\begin{bmatrix} 1 & 0 \\ 0 & b\end{bmatrix}\begin{bmatrix} \sqrt{a} & 0 \\ 0 & 1\end{bmatrix}\right) \\
& = \Phi\left(\begin{bmatrix} \sqrt{a} & 0 \\ 0 & 1\end{bmatrix}
\begin{bmatrix} b & 0 \\ 0 & 1\end{bmatrix}\begin{bmatrix} \sqrt{a} & 0 \\ 0 & 1\end{bmatrix}\right) = 
\Phi\left(\begin{bmatrix} ab & 0 \\ 0 & 1\end{bmatrix}\right) = \Phi\left(\begin{bmatrix} \det A & 0 \\ 0 & 1\end{bmatrix}\right) . \end{align*}
Next notice that $\Phi(-I)$ is an involution which commutes with any $\Phi(A)$ since
$$ \Phi(A) = \Phi((-I)A(-I)) = \Phi(-I)\Phi(A)\Phi(-I) ,$$
and by multiplying this equation by $\Phi(-I)$ we obtain $\Phi(A)\Phi(-I) = \Phi(-I)\Phi(A)$.
If $A$ is negative definite, it is similar to some matrix $\begin{bmatrix} a & 0 \\ 0 & b\end{bmatrix}$ with $a, b <0$. So 
\begin{align*}
 \Phi(A) &= \Phi\left(\begin{bmatrix} a & 0 \\ 0 & b\end{bmatrix}\right) = 
\Phi\left(\begin{bmatrix} \sqrt{-a} & 0 \\ 0 & \sqrt{-b}\end{bmatrix}
(-I)\begin{bmatrix} \sqrt{-a} & 0 \\ 0 & \sqrt{-b}\end{bmatrix}\right) \\ 
 &= \Phi\left(\begin{bmatrix} \sqrt{ab} & 0 \\ 0 & 1\end{bmatrix}\right) \Phi(-I)
\Phi\left(\begin{bmatrix} \sqrt{ab} & 0 \\ 0 & 1\end{bmatrix}\right) = 
\Phi(-I)\Phi\left(\begin{bmatrix} \det A & 0 \\ 0 & 1\end{bmatrix}\right) 
\end{align*}
and the proof is complete.
\end{proof}

\begin{Proposition}  \label{thm4.4}
Let $\Phi:\s_2(\C) \to \s_2(\C)$ be a regular J.T.P. homomorphism with 
$\Phi(J) \in \{ -I,I \}$.
Then there exist a unitary matrix $U$, 
unital multiplicative maps $\psi_1, \psi_2: [0,\infty) \to [0,\infty)$ with $\psi_i(0)=0$ for $i\in \{1,2\}$, and maps $\eta_1, \eta_2 :\{0, 1, 2\} \to \{-1,1\}$ which satisfy
$\eta_1(2) = \eta_2(2) = 1$ and $\eta_1(1) = \eta_2(1)$, so that $\Phi(A)$ has the form
$$\Phi(A) = U \left[
\begin{array}{cc}
\psi_1(|\det A|) \eta_1(\Syl(A)) & 0 \\
0 & \psi_2(|\det A|) \eta_2(\Syl(A)) 
\end{array} \right] U^*,$$
for every $A \in \s_2(\C)$, where $\Syl(A)$ is the inertia of $A$.
\end{Proposition}

\begin{proof}
Consider all matrices of the form $\begin{bmatrix} x & 0 \\ 0 & 1\end{bmatrix} \in \s_2(\C)$. They are isomorphic to the semigroup
of real numbers for multiplication, so $\Phi$ induces a J.T.P. homomorphism from $\RR$ to $\s_2(\C)$. From Proposition \ref{C_to_all} we know its form
and by previous Lemma it follows that
there exist a unitary matrix $U$, a diagonal matrix $D$ with $\pm 1$'s on its diagonal and 
multiplicative maps $\varphi_1, \varphi_2: \RR \to \RR$, such that
$$\Phi(A) = U D\left[
\begin{array}{cc}
\varphi_1(\det A) & 0 \\
0 & \varphi_2(\det A) 
\end{array} \right] U^*$$
for every positive definite or nondefinite matrix $A \in \s_2(\C)$. 
This can be written in the form
$$\Phi(A) = U \left[
\begin{array}{cc}
\psi_1(|\det A|) \eta_1(\Syl(A)) & 0 \\
0 & \psi_2(|\det A|) \eta_2(\Syl(A)) 
\end{array} \right] U^*,$$
where $\psi_1, \psi_2: [0,\infty) \to [0,\infty)$ are multiplicative maps, and $\Syl(A)$ is the inertia of $A$.
Since $\Phi(I) = I$, we obtain $\eta_1(2) = \eta_2(2) = 1$, and since $\Phi$ maps a nontrivial involution to a scalar, we obtain
$\eta_1(1) = \eta_2(1)$. 

We now have to prove this form also for negative definite matrices.
If $\psi_1(x) = \psi_2(x)$ for every $x \ge 0$, then $\Phi(A)$ is scalar for every positive definite or 
nondefinite matrix $A \in \s_2(\C)$. In this case matrix $U$ is still arbitrary. There exists a unitary matrix $U$ and a diagonal matrix 
$D$ with $\pm 1$'s on its diagonal, so that $\Phi(-I) = U D U^*$.

On the other hand, if $\psi_1(x) \ne \psi_2(x)$ for some $x \ge 0$, then $\Phi(-I)$ commutes with $\Phi \left(\begin{bmatrix} x & 0 \\ 0 & 1\end{bmatrix}\right)$
by previous Lemma and again $\Phi(-I) = U D U^*$. Now let $\eta_1(0)$ and $\eta_2(0)$ be defined by diagonal entries of matrix $D$. Every negative definite matrix $A \in \s_2(\C)$ can be written in the form
$A = \sqrt{-A}(-I)\sqrt{-A}$, so
\begin{align*}
\Phi(A) & = \Phi(\sqrt{-A})\Phi(-I)\Phi(\sqrt{-A}) \\
& = U \left[
\begin{array}{cc}
\psi_1(\sqrt{\det A}) & 0 \\
0 & \psi_2(\sqrt{\det A}) 
\end{array} \right]\left[
\begin{array}{cc}
\eta_1(0) & 0 \\
0 & \eta_2(0)
\end{array} \right] \cdot \\ 
&\hspace{4cm}\cdot\left[
\begin{array}{cc}
\psi_1(\sqrt{\det A}) & 0 \\
0 & \psi_2(\sqrt{\det A}) 
\end{array} \right] U^* \\
 &= U \left[
\begin{array}{cc}
\psi_1(|\det A|) \eta_1(\Syl(A)) & 0 \\
0 & \psi_2(|\det A|) \eta_2(\Syl(A)) 
\end{array} \right] U^*,
\end{align*}
which completes the proof.
\end{proof}

\medskip

The case when a nontrivial idempotent is mapped to a scalar is covered by the form (i) of Theorem \ref{main}.

\section{Nondegenerate case}

In this section we assume that for a regular J.T.P. homomorphism $\Phi:\s_2(\C) \to \s_2(\C)$ there exists $A\in \s_2(\C)$ with $\rank A=1$ such that $\Phi(A) \neq 0$. We refer to such regular $\Phi$ as {\it nondegenerate} J.T.P. homomorphism.

From Lemmas \ref{podobne_v_enake} and \ref{rang_1_v_0} it follows
that nontrivial involutions cannot be mapped to scalar matrices. Thus
$$\sigma\left(\Phi \left(\begin{bmatrix} 0 & 1 \\ 1 & 0 \end{bmatrix}\right) \right)= \{-1,1\}.$$

First lemma shows that rank 1 matrices are mapped to rank 1 matrices.

\begin{Lemma} \label{rank1-v-rank1}
Let $\Phi: \s_2(\C) \to \s_2(\C)$ be a nondegenerate J.T.P. homomorphism. Then
$\rank \Phi(A)=1$ for every $A\in \s_2(\C)$ with $\rank A=1$.
\end{Lemma}

\begin{proof}
By assumption there exists $A\in \s_2(\C)$ with $\rank A=1$, such that $\Phi(A) \neq 0$. Say that $\sigma(A)=\{0,a \}$ for some $a\in \RR^*$. Then $\Phi$ maps all matrices with such spectrum  to nonzero matrices.

Take an arbitrary $b\in \RR^*$. Then $$\begin{bmatrix} b & 0 \\ 0 & 0 \end{bmatrix} \begin{bmatrix} 0 & 0 \\ 0 & a \end{bmatrix} \begin{bmatrix} b & 0 \\ 0 & 0 \end{bmatrix} =0.$$
Since $\Phi \left( \begin{bmatrix} 0 & 0 \\ 0 & a \end{bmatrix} \right)$ is nonzero, $\Phi \left( \begin{bmatrix} b & 0 \\ 0 & 0 \end{bmatrix} \right)$ cannot be invertible. We need to show that it is nonzero.
Since
$$\Phi \left( \begin{bmatrix} a & 0 \\ 0 & 0 \end{bmatrix} \right)^2 = \Phi \left( \begin{bmatrix} \frac{a}{b} & 0 \\ 0 & 0 \end{bmatrix} \right) \Phi \left( \begin{bmatrix} b & 0 \\ 0 & 0 \end{bmatrix} \right)^2 \Phi \left( \begin{bmatrix} \frac{a}{b} & 0 \\ 0 & 0 \end{bmatrix} \right) \neq 0,$$
it follows that $\Phi \left( \begin{bmatrix} b & 0 \\ 0 & 0 \end{bmatrix} \right) \neq 0$.

Take $A\in \s_2(\C)$ an arbitrary matrix of rank 1. Then $A=B \begin{bmatrix} b & 0 \\ 0 & 0 \end{bmatrix} B$ for some $b\in \RR^*$ and some involution $B\in \s_2(\C)$. Hence, $\rank \Phi(A)= \rank \Phi \left( \begin{bmatrix} b & 0 \\ 0 & 0 \end{bmatrix} \right) =1$.
\end{proof}



\begin{Lemma} \label{lema5.1.5}
Let $\Phi:\s_2(\C) \to \s_2(\C)$ be a regular J.T.P. homomorphism. If there exists $\lambda \in\RR$ such that $\Phi(\lambda I)$ is not a scalar, then there exists a unitary matrix $U$ such that
$$\Phi(A)=U \begin{bmatrix} \varphi_1(A) & 0 \\ 0 & \varphi_2(A) \end{bmatrix} U^* \qquad \textrm{for all } A\in \s_2(\C),$$
where $\varphi_1,\varphi_2: \s_2(\C) \to \RR$ are distinct unital J.T.P. homomorphisms.
\end{Lemma}

\begin{proof}
Suppose there exists $\lambda_0 >0$ such that $\Phi(\lambda_0 I)$ is not a scalar matrix. Then $\Phi(\lambda_0 I)= U \begin{bmatrix} \alpha & 0 \\ 0 & \beta \end{bmatrix} U^*$ for some unitary matrix $U$ and $\alpha \neq \beta$. Taking the similarity action if necessary, we may assume without the loss of generality that $\Phi(\lambda_0 I)=\begin{bmatrix} \alpha & 0 \\ 0 & \beta \end{bmatrix}$. Because $\Phi(\lambda_0 I)=\Phi( \sqrt{\lambda_0} I)^2$, $\Phi(\lambda_0 I)$ is a positive definite matrix, hence $\alpha, \beta >0$.
Choose an arbitrary $A\in \s_2(\C)$ with $\Phi(A)=\begin{bmatrix} a & b \\ \bar{b} & c \end{bmatrix}$. Then
\begin{align*}
\Phi(A \lambda_0^2 I A) &= \begin{bmatrix} a & b \\ \bar{b} & c \end{bmatrix} \begin{bmatrix} \alpha & 0 \\ 0 & \beta \end{bmatrix} \begin{bmatrix} \alpha & 0 \\ 0 & \beta \end{bmatrix} \begin{bmatrix} a & b \\ \bar{b} & c \end{bmatrix} \\
&= \begin{bmatrix} \alpha^2 a^2 + \beta^2 |b|^2 & \alpha^2 a b+ \beta^2 b c \\ \alpha^2 a \bar{b}+ \beta^2 \bar{b} c & \alpha^2 |b|^2+\beta^2 c \end{bmatrix} \\
= \Phi(\lambda_0 I A^2 \lambda_0 I) & = \begin{bmatrix} \alpha & 0 \\ 0 & \beta \end{bmatrix} \begin{bmatrix} a & b \\ \bar{b} & c \end{bmatrix} \begin{bmatrix} a & b \\ \bar{b} & c \end{bmatrix} \begin{bmatrix} \alpha & 0 \\ 0 & \beta \end{bmatrix} \\
&= \begin{bmatrix} \alpha^2 a^2+ \alpha^2 |b|^2 & \alpha \beta a b + \alpha \beta b c \\ \alpha \beta a \bar{b}+ \alpha \beta \bar{b} c & \beta^2 |b|^2 +\beta^2 c^2 \end{bmatrix}.
\end{align*}
Equating upper left entries, we get $\alpha^2 |b|^2= \beta^2 |b|^2$. Note that $\alpha \neq \beta$ and $\alpha, \beta >0$, hence $b=0$.

We conclude that $\Phi(A)$ is diagonal for every $A\in\s_2(\C)$. Thus there exist distinct J.T.P. homomorphisms $\varphi_1,\varphi_2: \s_2(\C) \to \RR$ such that
$$\Phi(A)=\begin{bmatrix} \varphi_1(A) & 0 \\ 0 & \varphi_2(A) \end{bmatrix}.$$

We may now assume that $\Phi(\lambda I)$ is scalar for every $\lambda>0$.
To finish the proof, we may assume that $\Phi(-I)$ is not a scalar matrix. If that is not the case, then $\Phi(\lambda I)= \Phi(\sqrt{-\lambda}I) \Phi(-I) \Phi(\sqrt{-\lambda}I)$ is scalar also for every $\lambda < 0$. So suppose that $\Phi(-I)=U \begin{bmatrix} -1 & 0 \\ 0 & 1 \end{bmatrix}U^*$ for some unitary matrix $U$.
Since $$ \Phi(A) = \Phi((-I)A(-I)) = \Phi(-I)\Phi(A)\Phi(-I) ,$$
$\Phi(-I)$ commutes with any $\Phi(A)$, so again $\Phi(A)=U\begin{bmatrix} \varphi_1(A) & 0 \\ 0 & \varphi_2(A) \end{bmatrix} U^*$.
\end{proof}

Next we show that $\Phi$ maps scalar matrices to scalar matrices.
\begin{Lemma}
Let $\Phi:\s_2(\C) \to \s_2(\C)$ be a nondegenerate J.T.P. homomorphism. Then there exists a multiplicative map $\Psi: \RR \to \RR$ with $\Psi(1)=1$ such that
$\Phi(\lambda I)= \Psi(\lambda) I$ for every $\lambda \in \RR$.
\end{Lemma}

\begin{proof}
Suppose there exists $\lambda \in \RR$ such that $\Phi(\lambda I)$ is not a scalar matrix.
By previous lemma we know that
$$\Phi(A)=U\begin{bmatrix} \varphi_1(A) & 0 \\ 0 & \varphi_2(A) \end{bmatrix}U^*$$
for every $A \in \s_2(\C)$.
But then $\Phi(A)=0$ for every $A\in\s_2(\C)$ with $\rank A=1$ by Lemma \ref{rank-less-n}, which contradicts Lemma \ref{rank1-v-rank1}.
\end{proof}

A matrix $\begin{bmatrix} 1 & 0 \\ 0 & 0 \end{bmatrix}$ is an idempotent of $\rank 1$, hence it is mapped to an idempotent of rank 1 by Lemma \ref{rank1-v-rank1}. So $\Phi\left(\begin{bmatrix} 1 & 0 \\ 0 & 0 \end{bmatrix}\right)=U \begin{bmatrix} 1 & 0 \\ 0 & 0 \end{bmatrix}U^*$ for some $U\in \s_2(\C)$ unitary. By taking $\Phi'(A)=U\Phi(A)U^*$,  we may assume without the loss of generality that $\Phi \left(\begin{bmatrix} 1 & 0 \\ 0 & 0 \end{bmatrix}\right)=\begin{bmatrix} 1 & 0 \\ 0 & 0 \end{bmatrix}$. In other words, $\Phi$ preserves $E_{11}=\begin{bmatrix} 1 & 0 \\ 0 & 0 \end{bmatrix}$.

\begin{Lemma} \label{lema52}
Let $\Phi: \s_2(\C) \to \s_2(\C)$ be a nondegenerate J.T.P. homomorphism preserving $E_{11}$. Then
$\Phi \left(\begin{bmatrix} a & b \\ \bar{b} & c\end{bmatrix}\right) = \begin{bmatrix} \Psi(a) & * \\ * & * \end{bmatrix}$ for every $a\in \RR$.
\end{Lemma}

\begin{proof} To compute upper left entry of $\Phi\left( \begin{bmatrix} a & b \\ \bar{b} & c \end{bmatrix} \right)$, we take
\begin{align*}
\begin{bmatrix} 1 & 0 \\ 0 & 0 \end{bmatrix} \Phi \left(\begin{bmatrix} a & b \\ \bar{b} & c \end{bmatrix} \right) \begin{bmatrix} 1 & 0 \\ 0 & 0 \end{bmatrix}  &= \Phi \left( \begin{bmatrix} 1 & 0 \\ 0 & 0 \end{bmatrix} \begin{bmatrix} a & b \\ \bar{b} & c \end{bmatrix} \begin{bmatrix} 1 & 0 \\ 0 & 0 \end{bmatrix} \right)=\Phi \left(\begin{bmatrix} a & 0 \\ 0 & 0 \end{bmatrix} \right) \\
& = \Phi \left( \begin{bmatrix} 1 & 0 \\ 0 & 0 \end{bmatrix} \begin{bmatrix} a & 0 \\ 0 & a \end{bmatrix} \begin{bmatrix} 1 & 0 \\ 0 & 0 \end{bmatrix} \right) \\
&= \begin{bmatrix} 1 & 0 \\ 0 & 0 \end{bmatrix} \Psi(a) I \begin{bmatrix} 1 & 0 \\ 0 & 0 \end{bmatrix} 
 = \begin{bmatrix} \Psi(a)  & 0 \\ 0 & 0 \end{bmatrix},
\end{align*}
which concludes the proof.
\end{proof}

\begin{Lemma} \label{lema53}
Let $\Phi: \s_2(\C) \to \s_2(\C)$ be a nondegenerate J.T.P. homomorphism preserving $E_{11}$. Then
$\Phi \left(\begin{bmatrix} a & b \\ \bar{b} & c \end{bmatrix} \right)= \begin{bmatrix} * & * \\ * & \Psi(c) \end{bmatrix}$ for every $c\in \RR$.
\end{Lemma}

\begin{proof}
Since a matrix $\begin{bmatrix} 0 & 0 \\ 0 & 1 \end{bmatrix}$ is an idempotent of rank 1, it is mapped to an idempotent of rank 1. By previous lemma $\Phi \left(\begin{bmatrix} 0 & 0 \\ 0 & 1 \end{bmatrix} \right)$ has upper left entry equal to 0. Since its trace is 1, lower right entry equals 1, thus off-diagonal entries are 0. Consequently, $\Phi$ preserves $\begin{bmatrix} 0 & 0 \\ 0 & 1 \end{bmatrix}$ also and we proceed the  same way as in the proof of Lemma \ref{lema52}.
\end{proof}

\medskip

A consequence of Lemmas \ref{lema52} and \ref{lema53} is the equality $$\Phi \left(\begin{bmatrix} a & b \\ \bar{b} & c \end{bmatrix}\right) = \begin{bmatrix} \Psi(a) & * \\ * & \Psi(c) \end{bmatrix}.$$
Suppose
$\Phi \left(\begin{bmatrix} a & 0 \\ 0 & b \end{bmatrix} \right)= \begin{bmatrix} \Psi(a) & \beta \\ \bar{\beta} & \Psi(b) \end{bmatrix}$.
Then
$$\Phi \left(\begin{bmatrix} a & 0 \\ 0 & b \end{bmatrix} \right)^2= \Phi \left(\begin{bmatrix} a^2 & 0 \\ 0 & b^2 \end{bmatrix}\right) = \begin{bmatrix} \Psi(a)^2+|\beta|^2 & * \\ * & * \end{bmatrix} = \begin{bmatrix} \Psi(a^2) & * \\ * & * \end{bmatrix},$$
hence $|\beta|^2 =0$, which amounts to $\beta =0$. We conclude that $$\Phi\left(\begin{bmatrix} a & 0 \\ 0 & b \end{bmatrix}\right) = \begin{bmatrix} \Psi(a) & 0 \\ 0 & \Psi(b) \end{bmatrix}.$$

\begin{Lemma}
Let $\Phi: \s_2(\C) \to \s_2(\C)$ be a nondegenerate J.T.P. homomorphism preserving $E_{11}$. Then
$\Phi \left(\begin{bmatrix} a & 0 \\ 0 & b \end{bmatrix}\right) = \begin{bmatrix} \Psi(a) & 0 \\ 0 & \Psi(b) \end{bmatrix}$ for every $a,b\in \RR$.
\end{Lemma}

\medskip

Now take $x,y>0$. Then $\begin{bmatrix} x-y & \sqrt{x y} \\ \sqrt{x y} & 0 \end{bmatrix}= B \begin{bmatrix} x & 0 \\ 0 & -y \end{bmatrix} B$ for some Hermitian unitary matrix $B\in \s_2(\C)$ by Lemma \ref{lemma21}. Thus
$$\Phi \left(\begin{bmatrix} x-y & \sqrt{x y} \\ \sqrt{x y} & 0 \end{bmatrix}\right) = \begin{bmatrix} \Psi(x-y) & * \\ * & 0 \end{bmatrix} = \Phi(B) \begin{bmatrix} \Psi(x) & 0 \\ 0 & \Psi(-y) \end{bmatrix} \Phi(B).$$
Hence $\Psi(x-y)= \Psi(x)+ \Psi(-y)$, since trace of matrix is preserved under similarity action. Taking $y=x$, we get $\Psi(-x)=-\Psi(x)$ for $x>0$, hence for all $x\in \RR$. But then the equality $\Psi(x-y)= \Psi(x)+ \Psi(-y)$ also holds for all $x, y \in \RR$. Taking $z=-y$, we obtain additivity of $\Psi$.
Since a multiplicative function $\Psi: \RR \to \RR$ is additive, it must be an identity by \cite[Theorem 1.10]{sahoo}.


\medskip

We collect these facts into the following lemma.

\begin{Lemma}
Let $\Phi:\s_2(\C) \to \s_2(\C)$ be a nondegenerate J.T.P. homomorphism preserving $E_{11}$. Then 
$\Phi(\lambda I)= \lambda I$ for every $\lambda \in \RR$.

\end{Lemma}

Suppose $\Phi:\s_2(\C) \to \s_2(\C)$ is a nondegenerate J.T.P. homomorphism preserving $E_{11}$.
By the previous two lemmas, we have $\Phi \left(\begin{bmatrix} a & b \\ \bar{b} & c \end{bmatrix}\right) = \begin{bmatrix} a & * \\ * & c \end{bmatrix}$. Define
$$\Phi \left(\begin{bmatrix} 1 & 1 \\ 1 & 0 \end{bmatrix}\right)= \begin{bmatrix} 1 & \gamma \\ \bar{\gamma} & 0 \end{bmatrix} \quad \textrm{and} \quad \Phi \left(\begin{bmatrix} 0 & 1 \\ 1 & 0 \end{bmatrix}\right) = \begin{bmatrix} 0 & \delta \\ \bar{\delta} & 0 \end{bmatrix}.$$ Take $a>0$. Then
$$\begin{bmatrix} \sqrt{a} & 0 \\ 0 & \frac{1}{\sqrt{a}} \end{bmatrix} \begin{bmatrix} 1 & 1 \\ 1 & 0 \end{bmatrix} \begin{bmatrix} \sqrt{a} & 0 \\ 0 & \frac{1}{\sqrt{a}} \end{bmatrix} = \begin{bmatrix} a & 1 \\ 1 & 0 \end{bmatrix},$$
hence $\Phi \left(\begin{bmatrix} a & 1 \\ 1 & 0 \end{bmatrix}\right) = \begin{bmatrix} a & \gamma \\ \bar{\gamma} & 0 \end{bmatrix}$. We apply $\Phi$ on both hand sides of
\begin{equation*} 
\begin{bmatrix} 1 & 1 \\ 1 & 0 \end{bmatrix} \begin{bmatrix} 0 & 1 \\ 1 & 0 \end{bmatrix} \begin{bmatrix} 1 & 1 \\ 1 & 0 \end{bmatrix} = \begin{bmatrix} 2 & 1 \\ 1 & 0 \end{bmatrix}.
\end{equation*}
to obtain
\begin{align*}
\Phi \left(\begin{bmatrix} 1 & 1 \\ 1 & 0 \end{bmatrix}\right) &\Phi \left(\begin{bmatrix} 0 & 1 \\ 1 & 0 \end{bmatrix}\right)  \Phi \left(\begin{bmatrix} 1 & 1 \\ 1 & 0 \end{bmatrix}\right) = \begin{bmatrix} 1 & \gamma \\ \bar{\gamma} & 0 \end{bmatrix} \begin{bmatrix} 0 & \delta \\ \bar{\delta}& 0 \end{bmatrix} \begin{bmatrix} 1 & \gamma \\ \bar{\gamma} & 0 \end{bmatrix} =\\
& = \begin{bmatrix} \gamma \bar{\delta}+ \bar{\gamma} \delta & \gamma^2 \bar{\delta} \\ \bar{\gamma}^2 \delta & 0 \end{bmatrix} = \Phi \left( \begin{bmatrix} 2 & 1 \\ 1 & 0 \end{bmatrix}\right) = \begin{bmatrix} 2 & \gamma \\ \bar{\gamma} & 0 \end{bmatrix}.
\end{align*}
Thus we get $\gamma^2 \bar{\delta} = \gamma$. Since $\begin{bmatrix} 0 & 1 \\ 1 & 0 \end{bmatrix}$ is an involution, so is $\Phi \left( \begin{bmatrix} 0 & 1 \\ 1 & 0 \end{bmatrix}\right) = \begin{bmatrix} 0 & \delta \\ \bar{\delta} & 0 \end{bmatrix}$. This gives us $|\delta|=1$, hence $\gamma = \delta$.

The next step is taking arbitrary $x,y,z\in \RR$, $y,z\neq 0$, such that $\sign y= \sign z$. Note that
$$\begin{bmatrix} x & y \\ y & z \end{bmatrix} = \begin{bmatrix} \frac{y}{z} & 1 \\ 1 & 0 \end{bmatrix} \begin{bmatrix} z & 0 \\ 0 & x-\frac{y^2}{z} \end{bmatrix} \begin{bmatrix} \frac{y}{z} & 1 \\ 1 & 0 \end{bmatrix},$$
hence
\begin{equation} \label{eq1}
\Phi \left(\begin{bmatrix} x & y \\ y & z \end{bmatrix}\right) = \begin{bmatrix} x & y \gamma \\ y \bar{\gamma} & z \end{bmatrix}.
\end{equation}

On the other hand,
$$\Phi \left(\begin{bmatrix} x & -y \\ -y & z \end{bmatrix}\right) = \Phi \left( \begin{bmatrix} -1 & 0 \\ 0 & 1 \end{bmatrix} \begin{bmatrix} x & y \\ y & z \end{bmatrix} \begin{bmatrix} -1 & 0 \\ 0 & 1 \end{bmatrix} \right) = \begin{bmatrix} x & -y \gamma \\ -y \bar{\gamma} & z \end{bmatrix},$$
hence the equation \eqref{eq1} holds for all $x,y,z\in \RR$.

\medskip

Denote with $\Gamma$ the unit circle of $\C$.
From Lemmas \ref{lema52} and \ref{lema53} we know that there exists $\omega: \Gamma \to \C$ such that $\Phi \left( \begin{bmatrix} 0 & \beta \\ \bar{\beta} & 0 \end{bmatrix} \right)= \begin{bmatrix} 0 & \omega(\beta) \\ \overline{\omega(\beta)} & 0 \end{bmatrix}$ for every $\beta \in \Gamma$. Since $\begin{bmatrix} 0 & \omega(\beta) \\ \overline{\omega(\beta)} & 0 \end{bmatrix}$ is an involution, it must be that $|\omega(\beta)|=1$, hence $\omega: \Gamma \to \Gamma$. Define $\rho: \Gamma \to \Gamma$ with $\rho(\beta)= \frac{\omega(\beta)}{\beta \gamma}$.
Then it holds that $\Phi \left( \begin{bmatrix} 0 & \beta \\ \bar{\beta} & 0 \end{bmatrix} \right)= \begin{bmatrix} 0 & \beta\gamma\rho(\beta) \\ \overline{\beta\gamma\rho(\beta)} & 0 \end{bmatrix}$.
Also, $\rho(1)=1$.

Take $\alpha \in \Gamma$ to obtain
\begin{align*}
\Phi\left( \begin{bmatrix} 0& \alpha \beta^2 \\ \overline{\alpha \beta^2} & 0 \end{bmatrix} \right) &= \Phi \left( \begin{bmatrix} 0 & \beta \\ \bar{\beta} & 0 \end{bmatrix} \begin{bmatrix} 0 & \bar{\alpha} \\ \alpha & 0 \end{bmatrix} \begin{bmatrix} 0 & \beta \\ \bar{\beta} & 0 \end{bmatrix} \right) \\
&= \begin{bmatrix}  0 & \alpha \beta^2\gamma\rho(\alpha \beta^2) \\ \overline{\alpha \beta^2\gamma\rho(\alpha \beta^2)} & 0 \end{bmatrix} \\ 
&= \begin{bmatrix} 0 & \beta\gamma\rho(\beta) \\ \overline{\beta\gamma\rho(\beta)} & 0 \end{bmatrix} 
\begin{bmatrix} 0 & \bar{\alpha}\gamma\rho(\bar{\alpha}) \\ \overline{\bar{\alpha}\gamma\rho(\bar{\alpha})} & 0 \end{bmatrix} 
\begin{bmatrix} 0 & \beta\gamma\rho(\beta) \\ \overline{\beta\gamma\rho(\beta)} & 0 \end{bmatrix}\\
& = \begin{bmatrix} 0 & \alpha \beta^2\gamma\rho(\beta)^2 \overline{ \rho(\bar{\alpha})} \\ \overline{ \alpha \beta^2 \gamma \rho(\beta)^2} \rho(\bar{\alpha}) & 0 \end{bmatrix},
\end{align*}
so
\begin{equation} \label{eq2}
\rho(\alpha \beta^2)= \overline{\rho(\bar{\alpha})} \rho(\beta)^2.
\end{equation}
If we insert $\alpha=1$ and arbitrary $\beta$, we get $\rho(\beta^2)=\rho(\beta)^2$. On the other hand, if we insert $\beta=1$ and arbitrary $\alpha$, we get $\overline{\rho(\bar{\alpha})} = \rho(\alpha)$. Using these two expressions on \eqref{eq2}, we get
$\rho(\alpha \beta^2) = 
\rho(\alpha) \rho(\beta)^2 =
\rho(\alpha) \rho(\beta^2)  .$
Denoting $\beta':= \beta^2$, we get 
\begin{equation} \label{eq3}
\rho(\alpha \beta')= \rho(\alpha) \rho(\beta') 
\end{equation}
 for every $\alpha, \beta' \in \Gamma$, thus a function $\rho$ is multiplicative.

Take arbitrary $x,z \in \RR$ and $y\in\C$. Write $y=|y| e^{i \phi}$ for $\phi \in [0,2\pi)$. Then 
\begin{align*}
\Phi &\left( \begin{bmatrix} x & y \\ \bar{y} & z \end{bmatrix}\right) = \Phi \left( 
\begin{bmatrix} 0 & e^{i\frac{\phi}{2}} \\ e^{-i \frac{\phi}{2}} & 0 \end{bmatrix} 
\begin{bmatrix} z & |y| \\ |y| & x \end{bmatrix} \begin{bmatrix} 0 & e^{i\frac{\phi}{2}} \\ e^{-i \frac{\phi}{2}} & 0 \end{bmatrix} \right) = \\
& = \begin{bmatrix} 0 & e^{i\frac{\phi}{2}}\gamma\rho(e^{i\frac{\phi}{2}}) \\ \overline{e^{i\frac{\phi}{2}}\gamma\rho(e^{i\frac{\phi}{2}})} & 0 \end{bmatrix} \begin{bmatrix} z & |y| \gamma \\ |y| \bar{\gamma} & x \end{bmatrix} \begin{bmatrix} 0 & e^{i\frac{\phi}{2}}\gamma\rho(e^{i\frac{\phi}{2}}) \\ \overline{e^{i\frac{\phi}{2}}\gamma\rho(e^{i\frac{\phi}{2}})} & 0 \end{bmatrix} \\
&= \begin{bmatrix} x & |y| e^{i \phi} \gamma\rho(e^{i\frac{\phi}{2}})\rho(e^{i\frac{\phi}{2}}) \\ |y|\overline{e^{i \phi} \gamma\rho(e^{i\frac{\phi}{2}})\rho(e^{i\frac{\phi}{2}})} & z\end{bmatrix}= \begin{bmatrix} x & y\gamma  \rho(e^{i\phi}) \\ \overline{y\gamma\rho(e^{i\phi})} & z\end{bmatrix},
\end{align*}
where the last equality holds due to \eqref{eq3}.


\medskip

Now, take $\beta \in \Gamma$ and calculate
$$\begin{bmatrix} 1 & 1 \\ 1 & 1 \end{bmatrix} \begin{bmatrix} 1 & \beta \\ \bar{\beta} & 0 \end{bmatrix} \begin{bmatrix} 1 & 1 \\ 1 & 1 \end{bmatrix} = \begin{bmatrix} 1+ \beta+ \bar{\beta} & 1+ \beta +\bar{\beta} \\ 1+ \beta +\bar{\beta} & 1+ \beta +\bar{\beta} \end{bmatrix}.$$
Applying $\Phi$ on both hand sides, we obtain
\begin{align*}
\begin{bmatrix} 1 & \gamma \\ \bar{\gamma} & 1 \end{bmatrix} & \begin{bmatrix} 1 & \beta\gamma\rho(\beta) \\ \overline{\beta\gamma\rho(\beta)} & 0 \end{bmatrix} \begin{bmatrix} 1 & \gamma \\ \bar{\gamma} & 1 \end{bmatrix} =  \\
&=\begin{bmatrix} 1+ \beta\rho(\beta) +\bar{\beta}\rho(\bar{\beta}) & \gamma(1+ \beta\rho(\beta) +\bar{\beta}\rho(\bar{\beta})) \\ 
\bar{\gamma}(1+ \beta\rho(\beta) +\bar{\beta}\rho(\bar{\beta})) & 1+ \beta\rho(\beta) +\bar{\beta}\rho(\bar{\beta}) \end{bmatrix} \\
&= \begin{bmatrix} 1 + \beta + \bar{\beta} & \gamma(1 + \beta + \bar{\beta}) \\ \bar{\gamma}(1 + \beta + \bar{\beta}) & 1 + \beta + \bar{\beta} \end{bmatrix},
\end{align*}
which gives us $\beta + \bar{\beta} = \beta\rho(\beta) + \bar{\beta}\rho(\bar{\beta})$. Then $\mathrm{Re} \,\beta = \mathrm{Re}\, \beta\rho(\beta)$, and since $|\beta|=|\rho(\beta)|=1$, it must be that $|\mathrm{Im} \, \beta| = |\mathrm{Im} \, \beta\rho(\beta)|$. Hence $\mathrm{Im} \, \beta\rho(\beta) = \pm \mathrm{Im} \, \beta$. Thus we have either $\beta\rho(\beta)=\beta$ or $\beta\rho(\beta)=\bar{\beta}$.
We obtain that either
$$ \Phi\left( \begin{bmatrix} x & y \\ \bar{y} & z \end{bmatrix}\right) =  \begin{bmatrix} x & y\gamma \\ \bar{y}\bar{\gamma} & z \end{bmatrix}= \begin{bmatrix} 1 & 0 \\ 0 & \bar{\gamma}  \end{bmatrix} \begin{bmatrix} x & y \\ \bar{y} & z \end{bmatrix} \begin{bmatrix} 1 & 0 \\ 0 & \gamma \end{bmatrix} $$
or
$$ \Phi\left( \begin{bmatrix} x & y \\ \bar{y} & z \end{bmatrix}\right) =  \begin{bmatrix} x & \bar{y}\gamma \\ y\bar{\gamma} & z \end{bmatrix} =  \begin{bmatrix} 1 & 0 \\ 0 & \bar{\gamma} \end{bmatrix} \begin{bmatrix} x & \bar{y} \\ y & z \end{bmatrix} \begin{bmatrix} 1 & 0 \\ 0 & \gamma \end{bmatrix}.$$
It is clear that these two forms of $\Phi$ cannot exist simultaneously, hence $\Phi$ always takes a single form for every matrix in $\s_2(\C)$.

\medskip

These findings give us the following lemma.

\begin{Lemma}
Let $\Phi: \s_2(\C) \to \s_2(\C)$ be a nondegenerate J.T.P. homomorphism preserving $E_{11}$. Then there exists a diagonal unitary matrix $U$ such that either $\Phi(A)=U A U^*$ or $\Phi(A)=U A^T U^*= U \bar{A} U^* $.
\end{Lemma}

The main result of this section characterizes nondegenerate regular J.T.P. homomorphisms on $\s_2(\C)$.

\begin{Proposition} \label{thm5.9}
Let $\Phi: \s_2(\C) \to \s_2(\C)$. The map $\Phi$ is a nondegenerate J.T.P. homomorphism if and only if there exists a unitary matrix $U$ such that
$$\Phi(A)= U A U^* \quad or \quad \Phi(A)= U A^T U^* = U \bar{A} U^*.$$
\end{Proposition}

The nondegenerate case is covered by the forms (ii) and (iii) of Theorem \ref{main}.

\section{Degenerate case}
In this section we consider regular J.T.P. homomorphisms $\Phi:\s_2(\C) \to \s_2(\C)$ such that $\Phi(A)=0$ for every $A\in\s_2(\C)$ with $\rank A\le 1$.
We refer to such regular $\Phi$ as {\it degenerate} J.T.P. homomorphism.

Further we assume that
$$\sigma\left( \Phi\left( \begin{bmatrix} 0 & 1 \\ 1 & 0 \end{bmatrix} \right) \right)= \{ -1,1\}.$$
The other possibility was already considered in Section 5.

\begin{Lemma}
If there exists $\lambda \in\RR$ such that $\Phi(\lambda I)$ is not a scalar, then there exist unitary matrix $U$, distinct unital multiplicative maps $\psi_1, \psi_2: [0,\infty) \to [0,\infty)$ with $\psi_i(0)=0$ for $i\in \{1,2\}$, and maps $\eta_1, \eta_2 :\{0, 1, 2\} \to \{-1,1\}$ which satisfy
$\eta_1(2) = \eta_2(2) = 1$ and $\eta_1(1) \neq \eta_2(1)$, so that $\Phi(A)$ has the form
$$\Phi(A) = U \left[
\begin{array}{cc}
\psi_1(|\det A|) \eta_1(\Syl(A)) & 0 \\
0 & \psi_2(|\det A|) \eta_2(\Syl(A)) 
\end{array} \right] U^* ,$$
for every $A \in \s_2(\C)$, where $\Syl(A)$ is the inertia of $A$.
\end{Lemma}

\noindent
\begin{Remark}
Notice that in this case we get a similar form as in the Proposition \ref{thm4.4}.
\end{Remark}

\begin{proof}
Take $\lambda \in \RR$ such that $\Phi(\lambda I)$ is not a scalar. Then we know by Lemma \ref{lema5.1.5} that
$$\Phi(A)=U \begin{bmatrix} \phi_1(A) & 0 \\ 0 & \phi_2(A) \end{bmatrix} U^*,$$
where $\phi_1,\phi_2: \s_2(\C) \to \RR$ are unital J.T.P. homomorphisms. By Proposition \ref{thm3.3}
$$\Phi(A) = U \left[
\begin{array}{cc}
\psi_1(|\det A|) \eta_1(\Syl(A)) & 0 \\
0 & \psi_2(|\det A|) \eta_2(\Syl(A)) 
\end{array} \right] U^* ,$$
with $\eta_1(2)=\eta_2(2)=1$, since $\Phi(I)=I$, and $\eta_1(1)\neq \eta_2(1)$, since
$$\sigma\left( \Phi\left( \begin{bmatrix} 0 & 1 \\ 1 & 0 \end{bmatrix} \right) \right)= \{ -1,1\},$$
which concludes the proof.
\end{proof}

If $\Phi(\lambda I)$ is a scalar matrix for every $\lambda \in \RR$, then there exists $\Psi:\RR\to \RR$ multiplicative such that $\Phi(\lambda I)=\Psi(\lambda)I$.
In the remainder of this section we assume that $\Phi(\lambda I)=\Psi(\lambda) I$. Due to regularity of $\Phi$ it holds that $\Psi(0)=0$ and $\Psi(1)=1$.

A set of matrices $\left\{ \begin{bmatrix} a & 0 \\ 0 & 1 \end{bmatrix} : \, a\in \RR^* \right\}$  is isomorphic to the group
of nonzero real numbers for multiplication, so $\Phi$ induces a J.T.P. homomorphism from $\RR^*$ to the set of invertible matrices in $\s_2(\C)$. By Lemma \ref{C_to_invertible} it then holds that
$$\Phi\left( \begin{bmatrix} a & 0 \\ 0 & 1 \end{bmatrix} \right) = U \begin{bmatrix} \alpha(a) & 0 \\ 0 & \beta(a) \end{bmatrix} U^*$$
for some unitary matrix $U$ and $\alpha, \beta: \RR^* \to \RR^*$ unital multiplicative maps.
Without the loss of generality we may assume that $$\Phi \left( \begin{bmatrix} a & 0 \\ 0 & 1 \end{bmatrix} \right) = \begin{bmatrix} \alpha(a) & 0 \\ 0 & \beta(a) \end{bmatrix}$$ with $\alpha(1)=\beta(1)=1$ and $\{\alpha(-1),\beta(-1)\} =\{-1,1\}$. We may also assume that $\alpha(-1)=-1$ and $\beta(-1)=1$, hence 
$$\Phi\left( \begin{bmatrix} -1 & 0 \\ 0 & 1 \end{bmatrix} \right)= \begin{bmatrix} -1 & 0 \\ 0 & 1 \end{bmatrix}.$$

\begin{Lemma} \label{lemma62}
Let $\Phi:\s_2(\C) \to \s_2(\C)$ be a degenerate J.T.P. homomorphism mapping scalars to scalars such that 
$\Phi\left( \begin{bmatrix} -1 & 0 \\ 0 & 1 \end{bmatrix} \right)= \begin{bmatrix} -1 & 0 \\ 0 & 1 \end{bmatrix}.$
Then
$$\Phi\left( \begin{bmatrix} 0 & 1 \\ 1 & 0 \end{bmatrix} \right)=
\begin{cases}
\pm \begin{bmatrix} 1 & 0 \\ 0 & -1 \end{bmatrix}; & \textrm{or}  \\
\begin{bmatrix} 0 & b \\ \bar{b} & 0 \end{bmatrix}; & |b|=1.
\end{cases} $$
\end{Lemma}

\begin{proof}
First we notice that $\begin{bmatrix} -1 & 0 \\ 0 & 1 \end{bmatrix}$ and $\begin{bmatrix} 1 & 0 \\ 0 & -1 \end{bmatrix}$ are similar, hence their images are also similar by Lemma \ref{lemma21}.
From $\begin{bmatrix} -1 & 0 \\ 0 & 1 \end{bmatrix} \begin{bmatrix} 1 & 0 \\ 0 & -1 \end{bmatrix} \begin{bmatrix} -1 & 0 \\ 0 & 1 \end{bmatrix}=\begin{bmatrix} 1 & 0 \\ 0 & -1 \end{bmatrix}$
it follows that
$$\begin{bmatrix} -1 & 0 \\ 0 & 1 \end{bmatrix} \Phi \left(\begin{bmatrix} 1 & 0 \\ 0 & -1 \end{bmatrix} \right)\begin{bmatrix} -1 & 0 \\ 0 & 1 \end{bmatrix}=\Phi \left(\begin{bmatrix} 1 & 0 \\ 0 & -1 \end{bmatrix}\right).$$
So $\Phi \left(\begin{bmatrix} 1 & 0 \\ 0 & -1 \end{bmatrix} \right)$ commutes with $\begin{bmatrix} -1 & 0 \\ 0 & 1 \end{bmatrix}$, thus it is diagonal and we conclude that $\Phi\left( \begin{bmatrix} 1 & 0 \\ 0 & -1 \end{bmatrix} \right)= \pm \begin{bmatrix} -1 & 0 \\ 0 & 1 \end{bmatrix}$.

A matrix $\begin{bmatrix} 0 & 1 \\ 1 & 0 \end{bmatrix}$ is an involution, hence its image is also an involution and by Lemma \ref{involution} it has the form
$$\Phi\left( \begin{bmatrix} 0 & 1 \\ 1 & 0 \end{bmatrix} \right) = \begin{bmatrix} \pm \sqrt{1-|b|^2} & b \\ \bar{b} & \mp \sqrt{1-|b|^2}\end{bmatrix}$$
for some $b\in\C$ with $|b|\leq 1$.
Thus
\begin{align*}
\Phi \left( \begin{bmatrix} 1 & 0 \\ 0 & -1 \end{bmatrix} \right) &=  \Phi \left( \begin{bmatrix} 0 & 1 \\ 1 & 0 \end{bmatrix} \right) \Phi\left( \begin{bmatrix} -1 & 0 \\ 0 & 1 \end{bmatrix}\right) \Phi \left( \begin{bmatrix} 0 & 1 \\ 1 & 0 \end{bmatrix} \right) \\
&= \begin{bmatrix} \pm \sqrt{1-|b|^2} & b \\ \bar{b} & \mp \sqrt{1-|b|^2} \end{bmatrix} \begin{bmatrix} -1 & 0 \\ 0 & 1 \end{bmatrix} \cdot \\
& \hspace{4cm} \cdot\begin{bmatrix} \pm \sqrt{1-|b|^2} & b \\ \bar{b} & \mp \sqrt{1-|b|^2} \end{bmatrix} \\
&= \begin{bmatrix} -1+2|b|^2  & \mp 2 b \sqrt{1-|b|^2} \\
\mp 2\bar{b} \sqrt{1-|b|^2} & 1-2|b|^2\end{bmatrix} = \pm \begin{bmatrix} 1 & 0 \\ 0 & -1\end{bmatrix}.
\end{align*}
Equating off-diagonal entries implies $b=0$ or $|b|=1$, which concludes the proof.
\end{proof}

\begin{Remark}
In the first case, where $b=0$, images of involutions $\begin{bmatrix} -1 & 0 \\ 0 & 1 \end{bmatrix}$ and $\begin{bmatrix} 0 & 1 \\ 1 & 0 \end{bmatrix}$ commute. In the second case, where $|b|=1$, images don't commute. We will consider these cases in subsections 6.1 and 6.2.
\end{Remark}

\medskip

\begin{Lemma} \label{eta}
For an invertible matrix $A\in \s_2(\C)$ define
$$\eta(A)= \left\{
\begin{array}{rl}
1; & A>0 \textrm{ or } A<>0 \\
-1; & A<0
\end{array} \right. .$$
If $\Phi: \s_2(\C)\to \s_2(\C) $ is a J.T.P.  homomorphism, then so is $$\Phi'(A)=\begin{cases} \eta(A) \Phi(A); &\det A \neq 0 \\ 0; & \det A=0 \end{cases} .$$
\end{Lemma}

\begin{proof}
We split the proof into several cases.
\begin{itemize}
	\item If $\det A=0$ or $\det B=0$, the equation is trivial.
	\item If $A>0$ or $A<>0$ and $B>0$ or $B<>0$, then $$\Phi'(ABA)=\Phi(ABA)=\Phi'(A) \Phi'(B) \Phi'(A).$$
	\item If $A<0$ and $B>0$ or $B<>0$, then $$\Phi'(ABA)=\Phi(ABA)=-\Phi(A)\Phi(B)(-\Phi(A))=\Phi'(A) \Phi'(B) \Phi'(A).$$
	\item If $A>0$ or $A<>0$ and $B<0$, then $$\Phi'(ABA)=-\Phi(ABA)=\Phi(A)(-\Phi(B))\Phi(A)=\Phi'(A) \Phi'(B) \Phi'(A).$$
	\item If $A<0$ and $B<0$, then $$\Phi'(ABA)=-\Phi(ABA)=-\Phi(A)(-\Phi(B))(-\Phi(A))=\Phi'(A) \Phi'(B) \Phi'(A).$$
\end{itemize}
This concludes the proof.
\end{proof}

\medskip

\subsection{Case $b=0$}

In this subsection we have the following assumptions (C1):
\begin{itemize}
	\item $\Phi:\s_2(\C) \to \s_2(\C)$ a regular J.T.P. homomorphism;
	\item $\Phi(A)=0$ for every $A \in \s_2(\C)$ with $\rank A \leq 1$;
	\item $\Phi(\lambda I) = \Psi(\lambda) I$ for some $\Psi: \RR \to \RR$ multiplicative with $\Psi(0)=0$;
	\item $\Phi \left( \begin{bmatrix} a & 0 \\ 0 & 1 \end{bmatrix} \right)= \begin{bmatrix} \alpha(a) & 0 \\ 0 & \beta(a) \end{bmatrix}$ for $\alpha, \beta: \RR \to \RR$ unital multiplicative maps with $\alpha(-1)=-1$ and $\beta(-1)=1$;
	\item $\Phi\left( \begin{bmatrix} 0 & 1 \\ 1 & 0 \end{bmatrix} \right)= \pm \begin{bmatrix} 1 & 0 \\ 0 & -1 \end{bmatrix}$.
\end{itemize}
We have
\begin{align*}
\Phi \left( \begin{bmatrix} 1 & 0 \\ 0 & a \end{bmatrix} \right) &=  \Phi \left( \begin{bmatrix} 0 & 1 \\ 1 & 0 \end{bmatrix} \right) \Phi\left( \begin{bmatrix} a & 0 \\ 0 & 1 \end{bmatrix}\right) \Phi \left( \begin{bmatrix} 0 & 1 \\ 1 & 0 \end{bmatrix} \right) \\
&= \pm \begin{bmatrix} 1 & 0 \\ 0 & -1 \end{bmatrix} \begin{bmatrix} \alpha(a) & 0 \\ 0 & \beta(a) \end{bmatrix}  \left(\pm \begin{bmatrix} 1 & 0 \\ 0 & -1 \end{bmatrix} \right) = \begin{bmatrix} \alpha(a) & 0 \\ 0 & \beta(a) \end{bmatrix}.
\end{align*}
Let $a>0$. It follows that
$$\Phi\left( \begin{bmatrix} a & 0 \\ 0 & a \end{bmatrix} \right) = \Psi(a) I = \Phi \left( \begin{bmatrix} \sqrt{a} & 0 \\ 0 & 1 \end{bmatrix} \begin{bmatrix} 1 & 0 \\ 0 & a \end{bmatrix} \begin{bmatrix} \sqrt{a} & 0 \\ 0 & 1 \end{bmatrix} \right)= \begin{bmatrix} \alpha(a)^2 & 0 \\ 0 & \beta(a)^2 \end{bmatrix}.$$
Thus $\alpha(a)^2=\beta(a)^2$, hence $\alpha(a)=\pm \beta(a)$.

\medskip

It also holds that
$$\begin{bmatrix} \alpha(a) & 0 \\ 0 & \beta(a) \end{bmatrix}= \begin{bmatrix} \alpha(\sqrt{a})^2 & 0 \\ 0 & \beta(\sqrt{a})^2 \end{bmatrix},$$
hence it follows that $\alpha(a)>0$ and $\beta(a)>0$, which in turn implies that $\alpha(a)=\beta(a)$ for all $a>0$. By initial assumptions it also holds that $\alpha(-1)=-1$ and $\beta(-1)=1$. Thus, for $a<0$, we have $\alpha(a)=-\alpha(|a|)$ and $\beta(a)=\beta(|a|)$. Hence, $\alpha(-x)=-\alpha(x)$ and $\beta(-x)=\beta(x)$ for all $x\in \RR$ with $\alpha(x)=\beta(x)>0$ for all $x>0$.

\medskip

We can conclude that
\begin{itemize}
	\item $\Phi\left( \begin{bmatrix} a & 0 \\ 0 & 1 \end{bmatrix} \right) = \begin{bmatrix} -\alpha(|a|) & 0 \\ 0 & \alpha(|a|) \end{bmatrix} = \Phi\left( \begin{bmatrix} 1 & 0 \\ 0 & a \end{bmatrix} \right)$ for $a<0$;
	\item $\Phi\left( \begin{bmatrix} a & 0 \\ 0 & b \end{bmatrix} \right) = \begin{bmatrix} \alpha(ab) & 0 \\ 0 & \alpha(ab) \end{bmatrix}$ for $a,b>0$;
	\item $\Phi\left( \begin{bmatrix} a & 0 \\ 0 & b \end{bmatrix} \right) = \begin{bmatrix} -\alpha(|ab|) & 0 \\ 0 & \alpha(|ab|) \end{bmatrix} = \begin{bmatrix} \alpha(ab) & 0 \\ 0 & \alpha(|ab|) \end{bmatrix}$ for $a<0$ and $b>0$ or $a>0$ and $b<0$.
\end{itemize}

\medskip

From $-I$ an involution and $\Phi(\lambda I)=\Psi(\lambda) I$ it follows that $\Phi(-I)=\pm I$. If $\Phi(-I)=-I$, define $\Phi'(A)=\eta(A)\Phi(A)$, so $\Phi'(-I)=I$. Hence we can assume without the loss of generality that $\Phi(-I)=I$. Thus
$$\Phi\left( \begin{bmatrix} a & 0 \\ 0 & b \end{bmatrix} \right)=\begin{bmatrix} \alpha(ab) & 0 \\ 0 & \alpha(|ab|) \end{bmatrix}$$
for every $a,b\in\RR$, which gives us the following lemma.

\begin{Lemma}
Under assumptions \textup{(C1)} and $\Phi(-I)=I$ there exists a unital odd multiplicative function $\alpha:\RR \to \RR$ with $\alpha(-1)=-1$
$$\Phi\left( \begin{bmatrix} a & 0 \\ 0 & b \end{bmatrix} \right) = \begin{bmatrix} \alpha(a b) & 0 \\ 0 & \alpha(|ab|) \end{bmatrix}$$
for all $a,b\in \RR$.
\end{Lemma}

By our assumptions we have $\Phi \left( \begin{bmatrix} 0 & 1 \\ 1 & 0 \end{bmatrix} \right) = \pm \begin{bmatrix} 1 & 0 \\ 0 & -1 \end{bmatrix}$.
If $\Phi \left( \begin{bmatrix} 0 & 1 \\ 1 & 0 \end{bmatrix} \right) = \begin{bmatrix} -1 & 0 \\ 0 & 1 \end{bmatrix}$, then we observe that 
\begin{equation} \label{eq-1}\frac{1}{\sqrt{2}}\begin{bmatrix} 1 & 1 \\ 1 & -1 \end{bmatrix} \begin{bmatrix} 1 & 0 \\ 0 & -1 \end{bmatrix} \frac{1}{\sqrt{2}}\begin{bmatrix} 1 & 1 \\ 1 & -1 \end{bmatrix} = \begin{bmatrix} 0 & 1 \\ 1 & 0 \end{bmatrix}, \end{equation}
thus
$$\Phi\left(\frac{1}{\sqrt{2}}\begin{bmatrix} 1 & 1 \\ 1 & -1 \end{bmatrix}\right)  \begin{bmatrix} -1 & 0 \\ 0 & 1 \end{bmatrix} \Phi\left( \frac{1}{\sqrt{2}}\begin{bmatrix} 1 & 1 \\ 1 & -1 \end{bmatrix}\right) = \begin{bmatrix} -1 & 0 \\ 0 & 1 \end{bmatrix},$$
which implies that the matrices $\Phi\left(\displaystyle{{\frac{1}{\sqrt{2}}}}\begin{bmatrix} 1 & 1 \\ 1 & -1 \end{bmatrix}\right)$ and $\begin{bmatrix} -1 & 0 \\ 0 & 1 \end{bmatrix}$ commute. This means that $\Phi\left(\displaystyle\frac{1}{\sqrt{2}}\begin{bmatrix} 1 & 1 \\ 1 & -1 \end{bmatrix}\right)$ is diagonal and since it is an involution, it holds that
$$\Phi\left(\displaystyle\frac{1}{\sqrt{2}}\begin{bmatrix} 1 & 1 \\ 1 & -1 \end{bmatrix}\right)=\pm \begin{bmatrix} -1 & 0 \\ 0 & 1 \end{bmatrix}.$$

On the other hand, if $\Phi \left( \begin{bmatrix} 0 & 1 \\ 1 & 0 \end{bmatrix} \right) = \begin{bmatrix} 1 & 0 \\ 0 & -1 \end{bmatrix}$, then 
equation \eqref{eq-1} still holds, hence
$$\Phi\left(\frac{1}{\sqrt{2}}\begin{bmatrix} 1 & 1 \\ 1 & -1 \end{bmatrix}\right)  \begin{bmatrix} -1 & 0 \\ 0 & 1 \end{bmatrix} \Phi\left( \frac{1}{\sqrt{2}}\begin{bmatrix} 1 & 1 \\ 1 & -1 \end{bmatrix}\right) = \begin{bmatrix} 1 & 0 \\ 0 & -1 \end{bmatrix},$$
which implies that the matrix $\Phi\left(\displaystyle\frac{1}{\sqrt{2}}\begin{bmatrix} 1 & 1 \\ 1 & -1 \end{bmatrix}\right)$ has the form $\begin{bmatrix} 0 & b \\ \bar{b} & 0 \end{bmatrix}$. Since $\Phi\left(\displaystyle\frac{1}{\sqrt{2}}\begin{bmatrix} 1 & 1 \\ 1 & -1 \end{bmatrix}\right)$ is an involution, it must be that $|b|=1$.

Denote $S=\displaystyle\frac{1}{\sqrt{2}}\begin{bmatrix} 1 & 1 \\ 1 & -1 \end{bmatrix}$. A calculation shows that
$$ S = \begin{bmatrix} 1 & 0 \\ 0 & -\frac{1}{3} \end{bmatrix} S \begin{bmatrix} 1 & 0 \\ 0 & \frac{1}{\sqrt{7}} \end{bmatrix} S \begin{bmatrix} \frac{-7+5 \sqrt{7}}{\sqrt{2}} & 0 \\ 0 & \frac{-7-5 \sqrt{7}}{\sqrt{2}} \end{bmatrix} S \begin{bmatrix} 1 & 0 \\ 0 & \frac{1}{\sqrt{7}} \end{bmatrix} S \begin{bmatrix} 1 & 0 \\ 0 & -\frac{1}{3} \end{bmatrix}. $$
Using $\Phi$, we get
\begin{align*}
\Phi(S) &= \begin{bmatrix} 0 & b \\ \bar{b} & 0 \end{bmatrix} \\
&= \begin{bmatrix} -\alpha(\frac{1}{3}) & 0 \\ 0 & \alpha(\frac{1}{3}) \end{bmatrix} \begin{bmatrix} 0 & b \\ \bar{b} & 0 \end{bmatrix} \begin{bmatrix} \alpha(\frac{1}{\sqrt{7}}) & 0 \\ 0 & \alpha(\frac{1}{\sqrt{7}}) \end{bmatrix} \begin{bmatrix} 0 & b \\ \bar{b} & 0 \end{bmatrix}\begin{bmatrix} -\alpha(63) & 0 \\ 0 & \alpha(63) \end{bmatrix}  \cdot \\
&  \hspace{1.2cm} \cdot \begin{bmatrix} 0 & b \\ \bar{b} & 0 \end{bmatrix} \begin{bmatrix} \alpha(\frac{1}{\sqrt{7}}) & 0 \\ 0 & \alpha(\frac{1}{\sqrt{7}}) \end{bmatrix} \begin{bmatrix} 0 & b \\ \bar{b} & 0 \end{bmatrix}\begin{bmatrix} -\alpha(\frac{1}{3}) & 0 \\ 0 & \alpha(\frac{1}{3}) \end{bmatrix}= \begin{bmatrix} -1 & 0 \\ 0 & 1 \end{bmatrix},
\end{align*}
which is a contradiction. Thus we have showed that
$$\Phi\left(\frac{1}{\sqrt{2}}\begin{bmatrix} 1 & 1 \\ 1 & -1 \end{bmatrix}\right)=\pm \begin{bmatrix} -1 & 0 \\ 0 & 1 \end{bmatrix}.$$

This enables us to prove the following lemma.

\begin{Lemma}
Under assumptions \textup{(C1)} and $\Phi(-I)=I$ we have
$$\Phi\left( \begin{bmatrix} a & b \\ b & a \end{bmatrix} \right) = \begin{bmatrix} \alpha(a^2-b^2) & 0 \\ 0 & \alpha(|a^2-b^2|) \end{bmatrix} \qquad \textrm{for all } a,b\in\RR.$$
\end{Lemma}

\begin{proof}
Decompose
\begin{align*}
\Phi\left( \begin{bmatrix} a & b \\ b & a \end{bmatrix} \right) &= \Phi\left(\frac{1}{\sqrt{2}}\begin{bmatrix} 1 & 1 \\ 1 & -1 \end{bmatrix}\right) \Phi\left( \begin{bmatrix} a+b & 0 \\ 0 & a-b \end{bmatrix} \right) \Phi\left(\frac{1}{\sqrt{2}}\begin{bmatrix} 1 & 1 \\ 1 & -1 \end{bmatrix}\right) \\
&= \pm \begin{bmatrix} -1 & 0 \\ 0 & 1 \end{bmatrix} \begin{bmatrix} \alpha(a^2-b^2) & 0 \\ 0 & \alpha(|a^2-b^2|) \end{bmatrix} \left( \pm \begin{bmatrix} -1 & 0 \\ 0 & 1 \end{bmatrix} \right) \\
&= \begin{bmatrix} \alpha(a^2-b^2) & 0 \\ 0 & \alpha(|a^2-b^2|) \end{bmatrix}
\end{align*}
for all $a, b\in\RR$.
\end{proof}

\begin{Lemma}
Under assumptions \textup{(C1)} and $\Phi(-I)=I$ we have
$$\Phi\left( \begin{bmatrix} x & y \\ y & z \end{bmatrix} \right) = \begin{bmatrix} \alpha(xz-y^2) & 0 \\ 0 & \alpha(|xz-y^2|) \end{bmatrix}$$
for  all $x,y,z\in \RR$ with $xz>0$.
\end{Lemma}

\begin{proof}
Decompose the matrix $ \begin{bmatrix} x & y \\ y & z \end{bmatrix}$ in the following manner:
\begin{align*}
\Phi\left( \begin{bmatrix} x & y \\ y & z \end{bmatrix} \right) &= \Phi\left( \begin{bmatrix} 1 & 0 \\ 0 & \sqrt{\frac{z}{x}} \end{bmatrix} \right) \Phi\left( \begin{bmatrix} x & y\sqrt{\frac{x}{z}} \\ y\sqrt{\frac{x}{z}} & x \end{bmatrix} \right) \Phi\left( \begin{bmatrix} 1 & 0 \\ 0 & \sqrt{\frac{z}{x}} \end{bmatrix} \right) \\
&= \begin{bmatrix} \alpha(\sqrt{\frac{z}{x}}) & 0 \\ 0 & \alpha(\sqrt{\frac{z}{x}}) \end{bmatrix}
 \begin{bmatrix} \alpha(x^2-y^2 \frac{x}{z}) & 0 \\ 0 & \alpha\left( \left| x^2- y^2 \frac{x}{z}\right|\right) \end{bmatrix} \cdot\\
& \hspace{6cm} \cdot   \begin{bmatrix} \alpha(\sqrt{\frac{z}{x}}) & 0 \\ 0 & \alpha(\sqrt{\frac{z}{x}}) \end{bmatrix} \\
&= \begin{bmatrix} \alpha(xz-y^2) & 0 \\ 0 & \alpha(| xz-y^2 |) \end{bmatrix},
\end{align*}
which concludes the proof.
\end{proof}

\begin{Lemma}
Under assumptions \textup{(C1)} and $\Phi(-I)=I$ we have $$\Phi\left( \begin{bmatrix} \pm\sqrt{1-a^2} & a \\ a & \mp\sqrt{1-a^2} \end{bmatrix} \right)= \begin{bmatrix} -1 & 0 \\ 0 & 1 \end{bmatrix}$$
for every $a\in (0,1)$.
\end{Lemma}

\begin{proof}
For every $a\in (0,1)$ it holds that
$\begin{bmatrix} \pm\sqrt{1-a^2} & a \\ a & \mp\sqrt{1-a^2} \end{bmatrix} = ABA,$ where $$A= \frac{1}{a^2} \begin{bmatrix} a^2-1 & \sqrt{1-a^2} \\ \sqrt{1-a^2} & a^2-1 \end{bmatrix}$$ and $$B= \frac{a}{1-a^2} \begin{bmatrix} (2\mp a) \sqrt{1-a^2} & 2-a^2 \\ 2-a^2 & (2\pm a) \sqrt{1-a^2} \end{bmatrix}.$$
Diagonal entries of $A$ and $B$ have the same sign, so
$$\Phi(A)= \begin{bmatrix} -\alpha ( \frac{1-a^2}{a^2} ) & 0 \\ 0 & \alpha ( \frac{1-a^2}{a^2}  )\end{bmatrix} \quad \textrm{and} \quad
\Phi(B)= \begin{bmatrix} -\alpha \left(\frac{a^4}{(1-a^2)^2}\right) & 0 \\ 0 & \alpha \left(\frac{a^4}{(1-a^2)^2}\right) \end{bmatrix}.$$
Hence
$$\Phi\left( \begin{bmatrix} \pm\sqrt{1-a^2} & a \\ a & \mp\sqrt{1-a^2} \end{bmatrix} \right)=  \begin{bmatrix} -1 & 0 \\ 0 & 1 \end{bmatrix},$$
which concludes the proof.
\end{proof}

Now take $\lambda \in \Gamma$. Then
$$\begin{bmatrix} 0 & \lambda \\ \bar{\lambda}& 0 \end{bmatrix} \begin{bmatrix} \pm \sqrt{1-a^2} & a \\ a & \mp \sqrt{1-a^2} \end{bmatrix} \begin{bmatrix} 0 & \lambda \\ \bar{\lambda} & 0 \end{bmatrix} = \begin{bmatrix} \mp \sqrt{1-a^2} & a \lambda^2 \\ a \bar{\lambda}^2 & \pm \sqrt{1-a^2} \end{bmatrix}.$$
It also holds that
$$ \begin{bmatrix} 0 & \lambda \\ \bar{\lambda}& 0 \end{bmatrix} \begin{bmatrix} -1 & 0 \\ 0 & 1 \end{bmatrix} \begin{bmatrix} 0 & \lambda \\ \bar{\lambda}& 0 \end{bmatrix}=\begin{bmatrix} 1 & 0 \\ 0 & -1 \end{bmatrix},$$
hence
$$\Phi\left( \begin{bmatrix} 0 & \lambda \\ \bar{\lambda}& 0 \end{bmatrix}\right) \begin{bmatrix} -1 & 0 \\ 0 & 1 \end{bmatrix} \Phi\left( \begin{bmatrix} 0 & \lambda \\ \bar{\lambda}& 0 \end{bmatrix}\right)=\begin{bmatrix} -1 & 0 \\ 0 & 1 \end{bmatrix},$$
which implies that
$$\Phi\left( \begin{bmatrix} 0 & \lambda \\ \bar{\lambda}& 0 \end{bmatrix}\right)=\pm \begin{bmatrix} -1 & 0 \\ 0 & 1 \end{bmatrix}.$$
We conclude that
$$ \Phi\left( \begin{bmatrix} \pm\sqrt{1-a^2} & a\lambda^2  \\ a\bar{\lambda}^2 & \mp\sqrt{1-a^2} \end{bmatrix} \right)= \begin{bmatrix} -1 & 0 \\ 0 & 1 \end{bmatrix}.$$

This amounts to the following lemma.
\begin{Lemma}
Under assumptions \textup{(C1)} and $\Phi(-I)=I$
$\Phi$ maps every involution into a diagonal involution.
\end{Lemma}

Now take an arbitrary $A\in \s_2(\C)$. We know by Lemma \ref{lemma21} that $A=BDB$ for some involution $B$ and some diagonal matrix $D$. Hence
$$\Phi(A)= \pm \begin{bmatrix} -1 & 0 \\ 0 & 1 \end{bmatrix} \Phi(D) \left( \pm \begin{bmatrix} -1 & 0 \\ 0 & 1 \end{bmatrix} \right)= \Phi(D)=\begin{bmatrix} \alpha(\det A) & 0 \\ 0 & \alpha(|\det A|)\end{bmatrix}.$$

Case $b=0$ amounts to the following proposition.

\begin{Proposition} \label{thm6.9}
Let $\Phi: \s_2(\C) \to \s_2(\C)$ be a regular J.T.P. homomorphism such that
\begin{itemize}
	\item $\Phi(A)=0$ for every $A\in \s_2(\C)$ with $\rank A \leq 1$;
	\item $\Phi$ maps scalars to scalars;
	\item images of $\begin{bmatrix} 0 & 1 \\ 1 & 0 \end{bmatrix}$ and $\begin{bmatrix} -1 & 0 \\ 0 & 1 \end{bmatrix}$ commute.
\end{itemize}
Then there exist a unitary matrix $U$ and $\alpha: \RR \to \RR$ a unital multiplicative map with $\alpha(-1)=-1$ such that
$$\Phi(A)= U \begin{bmatrix} \alpha(\det A) & 0 \\ 0 & \alpha(|\det A|) \end{bmatrix} U^* \qquad \textrm{for every } A\in \s_2(\C)$$
or
$$\Phi(A)=\eta(A)  U \begin{bmatrix} \alpha(\det A) & 0 \\ 0 & \alpha(|\det A|) \end{bmatrix} U^* \qquad \textrm{for every } A\in \s_2(\C)$$
where $\eta$ is the function defined in Lemma \ref{eta}.
\end{Proposition}

This case is covered by the form (i) of Theorem \ref{main}.

\subsection{Case $|b|=1$}

In this subsection we consider $\Phi$ as in Lemma \ref{lemma62} such that $\Phi\left( \begin{bmatrix} 0 & 1 \\ 1 & 0 \end{bmatrix} \right)=\begin{bmatrix} 0 & b \\ \bar{b} & 0 \end{bmatrix}$ for some $b\in \Gamma$.

\begin{Lemma}
Let $\Phi:\s_2(\C) \to \s_2(\C)$ be a degenerate J.T.P. homomorphism mapping scalars to scalars such that 
$\Phi$ maps $\begin{bmatrix} a & 0 \\ 0 & 1 \end{bmatrix}$ to a diagonal matrix for every $a\in\RR$.
Then there exists a unitary matrix $U$ such that $\Phi\left( \begin{bmatrix} 0 & 1 \\ 1 & 0 \end{bmatrix} \right)=U\begin{bmatrix} 0 & 1 \\ 1 & 0 \end{bmatrix} U^*$.
\end{Lemma}

\begin{proof}
Since $|b|=1$, it can be written as $b=e^{i \phi}$. Define $a=e^{i \frac{\phi}{2}}$ and 
$U=\begin{bmatrix} 0 & a \\ \bar{a} & 0 \end{bmatrix}$. Then
$$\Phi\left( \begin{bmatrix} 0 & 1 \\ 1 & 0 \end{bmatrix} \right)= \begin{bmatrix} 0 & b \\ \bar{b} & 0 \end{bmatrix} = \begin{bmatrix} 0 & a \\ \bar{a} & 0 \end{bmatrix} \begin{bmatrix} 0 & 1 \\ 1 & 0 \end{bmatrix}  \begin{bmatrix} 0 & a \\ \bar{a} & 0 \end{bmatrix}=U\begin{bmatrix} 0 & 1 \\ 1 & 0 \end{bmatrix} U^*.$$
This concludes the proof.
\end{proof}

Define $\Phi'=U^* \Phi U$.
Thus $\Phi' \left( \begin{bmatrix} 0 & 1 \\ 1 & 0 \end{bmatrix} \right) = \begin{bmatrix} 0 & 1 \\ 1 & 0 \end{bmatrix}$. Since $\Phi'$ preserves the other assumptions of the lemma, we can substitute $\Phi$ for $\Phi'$, if necessary, so we can safely assume that $\Phi \left( \begin{bmatrix} 0 & 1 \\ 1 & 0 \end{bmatrix} \right) = \begin{bmatrix} 0 & 1 \\ 1 & 0 \end{bmatrix}$.

We have $\Phi \left( \begin{bmatrix} a & 0 \\ 0 & 1 \end{bmatrix} \right)= \begin{bmatrix} \alpha(a) & 0 \\ 0 & \beta(a) \end{bmatrix}$ for $\alpha, \beta: \RR \to \RR$ unital multiplicative maps with $\alpha(-1)=-1$ and $\beta(-1)=1$. 
If $a\neq 0$, we write it as
$$\Phi \left( \begin{bmatrix} a & 0 \\ 0 & 1 \end{bmatrix} \right) = \beta(a) \begin{bmatrix} \frac{\alpha(a)}{\beta(a)} & 0 \\ 0 & 1 \end{bmatrix}.$$
Define $\gamma: \RR^* \to \RR^*$ with $\gamma(a)=\frac{\alpha(a)}{\beta(a)}$. Then $\gamma$ is a unital multiplicative map with $\gamma(-1)=-1$.

\medskip

We now have the following assumptions (C2):
\begin{itemize}
	\item $\Phi:\s_2(\C) \to \s_2(\C)$ a regular J.T.P. homomorphism;
	\item $\Phi(A)=0$ for every $A \in \s_2(\C)$ with $\rank A \leq 1$;
	\item $\Phi(\lambda I) = \Psi(\lambda) I$ for some $\Psi: \RR \to \RR$ multiplicative with $\Psi(0)=0$;
	\item $\Phi \left( \begin{bmatrix} a & 0 \\ 0 & 1 \end{bmatrix} \right)= \beta(a) \begin{bmatrix} \gamma(a) & 0 \\ 0 & 1 \end{bmatrix}$ for $\beta, \gamma: \RR^* \to \RR^*$ unital multiplicative maps with $\beta(-1)=1$ and $\gamma(-1)=-1$;
	\item $\Phi\left( \begin{bmatrix} 0 & 1 \\ 1 & 0 \end{bmatrix} \right)= \begin{bmatrix} 0 & 1 \\ 1 & 0 \end{bmatrix}$.
\end{itemize}

\begin{Lemma}
Under assumptions \textup{(C2)} it holds that
$$\Phi \left( \begin{bmatrix} 1 & 0 \\ 0 & a \end{bmatrix} \right) =  \beta(a) \begin{bmatrix} 1 & 0 \\ 0 & \gamma(a) \end{bmatrix}$$
for every $a\in \RR^*$.
\end{Lemma}

\begin{proof}
We have
\begin{align*} \Phi \left( \begin{bmatrix} 1 & 0 \\ 0 & a \end{bmatrix} \right) &= \Phi \left( \begin{bmatrix} 0 & 1 \\ 1 & 0 \end{bmatrix} \right) \Phi \left( \begin{bmatrix} a & 0 \\ 0 & 1 \end{bmatrix} \right) \Phi \left( \begin{bmatrix} 0 & 1 \\ 1 & 0 \end{bmatrix} \right) \\
& = \beta(a) \begin{bmatrix} 0 & 1 \\ 1 & 0 \end{bmatrix} \begin{bmatrix} \gamma(a) & 0 \\ 0 & 1 \end{bmatrix} \begin{bmatrix} 0 & 1 \\ 1 & 0 \end{bmatrix} =  \beta(a) \begin{bmatrix} 1 & 0 \\ 0 & \gamma(a) \end{bmatrix},
\end{align*}
which concludes the proof.
\end{proof}

We know that
\begin{multicols}{2}
\begin{itemize}
	\item $\Phi \left( \begin{bmatrix} 1 & 0 \\ 0 & 1 \end{bmatrix} \right) = \begin{bmatrix} 1 & 0 \\ 0 & 1 \end{bmatrix}$;
	\item $\Phi \left( \begin{bmatrix} -1 & 0 \\ 0 & 1 \end{bmatrix} \right) = \begin{bmatrix} -1 & 0 \\ 0 & 1 \end{bmatrix}$;
	\item $\Phi \left( \begin{bmatrix} 1 & 0 \\ 0 & -1 \end{bmatrix} \right) = \begin{bmatrix} 1 & 0 \\ 0 & -1 \end{bmatrix}$;
	\item $\Phi \left( \begin{bmatrix} -1 & 0 \\ 0 & -1 \end{bmatrix} \right) = \pm \begin{bmatrix} 1 & 0 \\ 0 & 1 \end{bmatrix}$.
\end{itemize}
\end{multicols}
If $\Phi(-I)=I$, we multiply $\Phi$ by $\eta$ from Lemma \ref{eta} to get $\Phi(-I)=-I$. So we may assume without the loss of generality that $\Phi(-I)=-I$.

\begin{Lemma}
Under assumptions \textup{(C2)} and $\Phi(-I)=-I$ we have
$$\Phi \left( \begin{bmatrix} a & 0 \\ 0 & b \end{bmatrix} \right) = \beta(a b)  \begin{bmatrix} \gamma(a) & 0 \\ 0 & \gamma(b) \end{bmatrix}$$ for every $a, b\in \RR$.
\end{Lemma}

\begin{proof}
First take $a>0$. Then
\begin{align*}
\Phi \left( \begin{bmatrix} a & 0 \\ 0 & b \end{bmatrix} \right) &= \Phi \left( \begin{bmatrix} \sqrt{a} & 0 \\ 0 & 1 \end{bmatrix} \right) \Phi \left( \begin{bmatrix} 1 & 0 \\ 0 & b \end{bmatrix} \right) \Phi \left( \begin{bmatrix} \sqrt{a} & 0 \\ 0 & 1 \end{bmatrix} \right) \\
&= \beta(\sqrt{a}) \begin{bmatrix} \gamma(\sqrt{a}) & 0 \\ 0 & 1 \end{bmatrix} \beta(b) \begin{bmatrix} 1 & 0 \\ 0 & \gamma(b) \end{bmatrix} \beta(\sqrt{a}) \begin{bmatrix} \gamma(\sqrt{a}) & 0 \\ 0 & 1 \end{bmatrix} \\
&= \beta(ab) \begin{bmatrix} \gamma(a) & 0 \\ 0 & \gamma(b)\end{bmatrix}.
\end{align*}
Similarly, we can prove the lemma for $b>0$.

If $a,b<0$, write
\begin{align*}
\Phi \left( \begin{bmatrix} a & 0 \\ 0 & b \end{bmatrix} \right) &= \Phi \left( \begin{bmatrix} \sqrt{|a|} & 0 \\ 0 & \sqrt{|b|} \end{bmatrix} \right) \Phi(-I) \Phi \left( \begin{bmatrix} \sqrt{|a|} & 0 \\ 0 & \sqrt{|b|} \end{bmatrix} \right) \\
& = - \beta(|ab|) \begin{bmatrix} \gamma(|a|) & 0 \\ 0 & \gamma(|b|)\end{bmatrix}= \beta(ab) \begin{bmatrix} \gamma(a) & 0 \\ 0 & \gamma(b)\end{bmatrix},
\end{align*}
which concludes the proof.
\end{proof}

\begin{Lemma}
Under assumptions \textup{(C2)} we have
$$\Phi \left( \frac{1}{\sqrt{2}}\begin{bmatrix} 1 & 1 \\ 1 & -1 \end{bmatrix} \right) = \pm \frac{1}{\sqrt{2}}\begin{bmatrix} 1 & 1 \\ 1 & -1 \end{bmatrix}.$$
\end{Lemma}

\begin{proof}
Since $\displaystyle\frac{1}{\sqrt{2}}\begin{bmatrix} 1 & 1 \\ 1 & -1 \end{bmatrix}$ is an involution, it is mapped to a nontrivial involution, hence it must be that
$$\Phi \left( \frac{1}{\sqrt{2}}\begin{bmatrix} 1 & 1 \\ 1 & -1 \end{bmatrix} \right) = \begin{bmatrix} \pm \sqrt{1-|a|^2} & a \\ \bar{a} & \mp \sqrt{1-|a|^2} \end{bmatrix}.$$
A short calculation shows that
$$\frac{1}{\sqrt{2}}\begin{bmatrix} 1 & 1 \\ 1 & -1 \end{bmatrix} \begin{bmatrix} 1 & 0 \\ 0 & -1 \end{bmatrix} \frac{1}{\sqrt{2}}\begin{bmatrix} 1 & 1 \\ 1 & -1 \end{bmatrix} = \begin{bmatrix} 0 & 1 \\ 1 & 0 \end{bmatrix}.$$
Applying $\Phi$ on both hand sides of the equation, we get
$$\begin{bmatrix} 1- 2|a|^2 & \pm 2a \sqrt{1-|a|^2} \\ \pm 2\bar{a} \sqrt{1-|a|^2} & 2|a|^2-1 \end{bmatrix} = \begin{bmatrix} 0 & 1 \\ 1 & 0 \end{bmatrix}.$$
This is possible only when $1-2|a|^2=0$ and $\pm 2a \sqrt{1-|a|^2}=1$. The first equation shows that $|a|=\frac{1}{\sqrt{2}}$, and the second then implies that $a=\pm \frac{1}{\sqrt{2}}$.
\end{proof}

\begin{Lemma} \label{lemma6.15}
Let assumptions \textup{(C2)} hold and take arbitrary $a,b\in\RR$. Then $\Phi \left( \begin{bmatrix} a+2b & b \\ b & a \end{bmatrix} \right) = \begin{bmatrix} a'+2b' & b' \\ b' & a' \end{bmatrix}$ for some $a',b'\in\RR$.
\end{Lemma}

\begin{proof}
A matrix $\begin{bmatrix} a+2b & b \\ b & a \end{bmatrix}$ commutes with $\displaystyle\frac{1}{\sqrt{2}}\begin{bmatrix} 1 & 1 \\ 1 & -1 \end{bmatrix}$, hence
$$\frac{1}{\sqrt{2}}\begin{bmatrix} 1 & 1 \\ 1 & -1 \end{bmatrix} \begin{bmatrix} a+2b & b \\ b & a \end{bmatrix}  \frac{1}{\sqrt{2}}\begin{bmatrix} 1 & 1 \\ 1 & -1 \end{bmatrix} = \begin{bmatrix} a+2b & b \\ b & a \end{bmatrix}.$$
Then $\Phi\left( \begin{bmatrix} a+2b & b \\ b & a \end{bmatrix} \right)$ commutes with $\Phi \left( \displaystyle\frac{1}{\sqrt{2}}\begin{bmatrix} 1 & 1 \\ 1 & -1 \end{bmatrix} \right) = \pm\displaystyle \frac{1}{\sqrt{2}}\begin{bmatrix} 1 & 1 \\ 1 & -1 \end{bmatrix}$, which implies that
matrix $\Phi\left( \begin{bmatrix} a+2b & b \\ b & a \end{bmatrix} \right)$ has the form $\begin{bmatrix} a'+2b' & b' \\ b' & a' \end{bmatrix}$ for some $a',b'\in\RR$.
\end{proof}

\begin{Lemma}
Under assumptions \textup{(C2)} the function $\gamma: \RR^+ \to \RR^+$ satisfies functional equation
\[ \gamma \left( \frac{x-1+\sqrt{2x^2+2}}{x+1} \right) = \frac{ \gamma(x)-1+ \sqrt{2 \gamma(x)^2+2}}{\gamma(x)+1} \label{eq:fe} \tag{f.e.}. \]
\end{Lemma}

\begin{proof}
Take $x,a>0$. Then
\begin{align*}
\Phi & \left( \begin{bmatrix} a & 0 \\ 0 & 1 \end{bmatrix} \frac{1}{\sqrt{2}}\begin{bmatrix} 1 & 1 \\ 1 & -1 \end{bmatrix} \begin{bmatrix} x & 0 \\ 0 & 1 \end{bmatrix} \frac{1}{\sqrt{2}}\begin{bmatrix} 1 & 1 \\ 1 & -1 \end{bmatrix}  \begin{bmatrix} a & 0 \\ 0 & 1 \end{bmatrix} \right)= \\
&= \beta(a^2 x) \begin{bmatrix} \gamma(a) & 0 \\ 0 & 1 \end{bmatrix}\! \left( \pm \frac{1}{\sqrt{2}}\begin{bmatrix} 1 & 1 \\ 1 & -1 \end{bmatrix} \right)\! \begin{bmatrix} \gamma(x) & 0 \\ 0 & 1 \end{bmatrix}  \left( \pm \frac{1}{\sqrt{2}}\begin{bmatrix} 1 & 1 \\ 1 & -1 \end{bmatrix} \right) \begin{bmatrix} \gamma(a) & 0 \\ 0 & 1 \end{bmatrix} \\
&= \beta(a^2 x) \begin{bmatrix} \gamma(a^2) \frac{\gamma(x)+1}{2} & \gamma(a) \frac{\gamma(x)-1}{2} \\ \gamma(a) \frac{\gamma(x)-1}{2} & \frac{ \gamma(x)+1}{2} \end{bmatrix} =
\Phi \left( \begin{bmatrix} a^2 \frac{x+1}{2} & a \frac{x-1}{2} \\ a \frac{x-1}{2} & \frac{x+1}{2} \end{bmatrix} \right).
\end{align*}
We would like the matrix $A=\begin{bmatrix} a^2 \frac{x+1}{2} & a \frac{x-1}{2} \\ a \frac{x-1}{2} & \frac{x+1}{2} \end{bmatrix}$ to have the form as in Lemma \ref{lemma6.15}, hence
choose $a \in \RR$ such that
$$a^2 \frac{x+1}{2}=\frac{x+1}{2}+2a \frac{x-1}{2}.$$
Taking for $a$ the positive solution of this quadratic equation, we get
$$a=\frac{x-1+\sqrt{2x^2+2}}{x+1}.$$
The matrix $A$ is therefore mapped to a matrix of the same form by Lemma \ref{lemma6.15}, hence
$$\gamma(a^2) \frac{\gamma(x)+1}{2}= 2 \gamma(a) \frac{\gamma(x)-1}{2}+\frac{\gamma(x)+1}{2}.$$
Since $a$ is positive, it is mapped by $\gamma$ to a positive solution of the new quadratic equation, thus
$$\gamma(a)=\gamma \left( \frac{x-1+\sqrt{2x^2+2}}{x+1} \right) = \frac{ \gamma(x)-1 +\sqrt{2 \gamma(x)^2+2}}{\gamma(x)+1},$$
which concludes the proof.
\end{proof}

\begin{Lemma}
Under assumptions \textup{(C2)} the function
$\gamma$ has one of the following forms:
$$\gamma(x)=x \quad \textrm{for every } x\in \RR^* \qquad \textrm{ or } \qquad \gamma(x)=x^{-1} \quad \textrm{for every } x\in \RR^*.$$
\end{Lemma}

\begin{proof}
Let us prove the lemma for $x>0$ first. For $x<0$ it will then follow, since $\gamma(-x)=-\gamma(x)$.

We know that $\gamma: \RR^+ \to \RR^+$ is a multiplicative function satisfying \eqref{eq:fe}. By \cite[Theorem 2.4]{sahoo} $\gamma$ has the form $\gamma(x)=e^{f( \log x)}$ for every $x>0$, where $f: \RR \to \RR$ is additive.
From \eqref{eq:fe} it follows that
$$e^{ f\left( \log \frac{x-1+\sqrt{2x^2+2}}{x+1} \right)} = \frac{ e^{f(\log x)}-1 + \sqrt{2e^{2f(\log x)}+2}}{e^{f(\log x)} +1},$$
hence
$$f\left( \log \frac{x-1+\sqrt{2x^2+2}}{x+1} \right) = \log \frac{ e^{f(\log x)}-1 + \sqrt{2e^{2f(\log x)}+2}}{e^{f(\log x)} +1}.$$
Taking $x\in (1,\infty)$, we get
$z=\log \frac{x-1+\sqrt{2x^2+2}}{x+1} \in (0, \log(1+\sqrt{2})).$
Substituting $y=f(\log x)$, we get
\begin{align*}
f(z) &= \log \frac{ e^{f(\log x)}-1 + \sqrt{2e^{2f(\log x)}+2}}{e^{f(\log x)} +1} = \log \frac{ e^y-1+ \sqrt{2e^{2y}+2}}{e^y+1}.
\end{align*}
Then for $t>0$ the following estimation manipulation
\begin{align*}
2t^2+2  & \le (\sqrt{2} t + 2+ \sqrt{2})^2 \\
\sqrt{2t^2+2} & \le \sqrt{2} t +1+ \sqrt{2} + 1 \\
t-1 + \sqrt{2t^2 +2} & \le (1+\sqrt{2})(t+1) \\
\frac{t-1+\sqrt{2t^2+2}}{t+1} & \le 1+\sqrt{2}
\end{align*}
shows that $f(z) \le \log (1+\sqrt{2})$.

Thus additive function $f$ is bounded on an open interval $(0, \log (1+\sqrt{2}))$, hence by \cite[Theorem 1.8]{sahoo} it is linear.
Since it has the form $f(z)=cz $ for some $c\in \RR$, it follows that $\gamma(x)=x^c$.

We get
$$\left( \frac{x-1+\sqrt{2x^2+2}}{x+1} \right)^c = \frac{ x^c-1+\sqrt{2x^{2c}+2}}{x^c+1}$$
for every $x>0$.
If $c>0$, by taking $\lim_{x \to \infty}$ we get $(1+\sqrt{2})^c= 1+\sqrt{2}$. Thus $c=1$. If $c<0$, again by taking $\lim_{x\to \infty}$ we get $(1+\sqrt{2})^c=-1+\sqrt{2}$, which implies $c=-1$. The last solution is $c=0$. By taking $c=0$, we get $\Phi\left(\begin{bmatrix} x & 0 \\ 0 & 1 \end{bmatrix} \right)=I$ for every $x>0$, which is a contradiction with our assumptions, hence $c \in \{-1,1\}$.
\end{proof}

First we show that in the case $c=-1$ one can reduce the proof to the case $c=1$. We have $\gamma(x)=\frac{1}{x}$. Under assumptions \textup{(C2)} and $\Phi(-I)=-I$ we have
$$\Phi \left( \begin{bmatrix} a & 0 \\ 0 & b \end{bmatrix} \right) = \beta(ab) \begin{bmatrix} \frac{1}{a} & 0 \\ 0 & \frac{1}{b} \end{bmatrix} = \frac{\beta(ab)}{ab} \begin{bmatrix} b & 0 \\ 0 & a \end{bmatrix}$$
for every $a,b \in \RR^*$.
Define
$\Phi'(A)=
\begin{cases}
\Phi(A^{-1}); & \rank A=2 \\
0; & \rank A \le 1
\end{cases}.$
Then, introducing new notation $\Psi'(t)=\Psi(t^{-1})$ and $\beta'(t)=\beta(t^{-1})$, we have
\begin{multicols}{2}
\begin{itemize}
	\item $\Phi'(0)=0$;
	\item $\Phi'(I)=I$;
	\item $\Phi'(-I)=-I$;
	\item $\Phi'(\lambda I)= \Psi(\lambda^{-1})I= \Psi'(\lambda)I$;
	\item $\Phi'\left( \begin{bmatrix} 0 & 1 \\ 1 & 0 \end{bmatrix} \right) = \begin{bmatrix} 0 & 1 \\ 1 & 0 \end{bmatrix}$;
\end{itemize}
\end{multicols}
\begin{itemize}
	\item $\Phi'\left( \begin{bmatrix} a & 0 \\ 0 & b \end{bmatrix} \right) = \Phi\left( \begin{bmatrix} \frac{1}{a} & 0 \\ 0 & \frac{1}{b} \end{bmatrix} \right)= \beta\left(\frac{1}{ab}\right) \begin{bmatrix} a & 0 \\ 0 & b \end{bmatrix}= \beta'(ab) \begin{bmatrix} a & 0 \\ 0 & b \end{bmatrix} $;
	\item $\Phi'\left(\displaystyle \frac{1}{\sqrt{2}}\begin{bmatrix} 1 & 1 \\ 1 & -1 \end{bmatrix} \right) = \Phi\left(\displaystyle \frac{1}{\sqrt{2}}\begin{bmatrix} 1 & 1 \\ 1 & -1 \end{bmatrix} \right) =\pm \displaystyle\frac{1}{\sqrt{2}}\begin{bmatrix} 1 & 1 \\ 1 & -1 \end{bmatrix}$;
	\item $\Phi' \left( \begin{bmatrix} a+2b & b \\ b & a \end{bmatrix} \right) = \begin{bmatrix} a'+2b' & b' \\ b' & a' \end{bmatrix}$.
\end{itemize}
So taking $\Phi'$ instead of $\Phi$, if necessary, we may assume that $\gamma(x)=x$.

\begin{Lemma}
Under assumptions \textup{(C2)}, $\Phi(-I)=-I$ and $\gamma(x)=x$ we have
$$\Phi \left( \begin{bmatrix} a & b \\ b & a \end{bmatrix} \right) = \beta(a^2-b^2) \begin{bmatrix} a & b \\ b & a \end{bmatrix}.$$
\end{Lemma}

\begin{proof}
We have
\begin{align*}
\Phi \left( \begin{bmatrix} a & b \\ b & a \end{bmatrix} \right) &= \Phi \left( \frac{1}{\sqrt{2}}\begin{bmatrix} 1 & 1 \\ 1 & -1 \end{bmatrix}  \begin{bmatrix} a+b & 0 \\ 0 & a-b \end{bmatrix} \frac{1}{\sqrt{2}}\begin{bmatrix} 1 & 1 \\ 1 & -1 \end{bmatrix} \right) \\
&= \pm \frac{1}{\sqrt{2}}\begin{bmatrix} 1 & 1 \\ 1 & -1 \end{bmatrix} \beta(a^2-b^2) \begin{bmatrix} a+b & 0 \\ 0 & a-b \end{bmatrix} \left( \pm \frac{1}{\sqrt{2}}\begin{bmatrix} 1 & 1 \\ 1 & -1 \end{bmatrix} \right) \\
&= \beta(a^2-b^2) \begin{bmatrix} a & b \\ b & a \end{bmatrix},
\end{align*}
which concludes the proof.
\end{proof}

\begin{Lemma}
Under assumptions \textup{(C2)}, $\Phi(-I)=-I$ and $\gamma(x)=x$ we have
$\Phi \left( \begin{bmatrix} a & b \\ b & c \end{bmatrix} \right) = \beta(ac-b^2) \begin{bmatrix} a & b \\ b & c \end{bmatrix}$ for every $a,b,c\in \RR$.
\end{Lemma}

\begin{proof}
First take $ac >0$. We have
\begin{align*}
\Phi \left( \begin{bmatrix} a & b \\ b & c \end{bmatrix} \right) &= \Phi \left( \begin{bmatrix} \sqrt{\frac{a}{c}} & 0 \\ 0 & 1 \end{bmatrix} \begin{bmatrix} c & b \sqrt{\frac{c}{a}} \\ b \sqrt{\frac{c}{a}} & c \end{bmatrix} \begin{bmatrix} \sqrt{\frac{a}{c}} & 0 \\ 0 & 1 \end{bmatrix} \right) \\
&= \beta\left(\frac{a}{c}\right) \beta \left( c^2-b^2 \frac{c}{a} \right) \begin{bmatrix} a & b \\ b & c \end{bmatrix}.
\end{align*}

Next take $a\neq 0$. Straightforward calculation shows that
\begin{align*}
\Phi \left( \begin{bmatrix} a & b \\ b & -a \end{bmatrix} \right) &= \Phi \left( \begin{bmatrix} \frac{1}{2} & \frac{b}{a} \\ \frac{b}{a} & 1+2\frac{b^2}{a^2} \end{bmatrix} \begin{bmatrix} 4a+4\frac{b^2}{a} & 0 \\ 0 & -a \end{bmatrix} \begin{bmatrix} \frac{1}{2} & \frac{b}{a} \\ \frac{b}{a} & 1+2\frac{b^2}{a^2} \end{bmatrix} \right) \\
&= \beta(-a^2-b^2) \begin{bmatrix} a & b \\ b & -a \end{bmatrix}.
\end{align*}
Thus every real involution maps to itself. Since every real symmetric $2 \times 2$ matrix can be written $BDB$, where $B$ is a real symmetric involution and $D$ diagonal, the assertion follows.
\end{proof}

\begin{Lemma}
Assume the set of assumptions \textup{(C2)}, $\Phi(-I)=-I$, and $\gamma \equiv \id$. For arbitrary $x\in \C$ with $|x|=1$ we have
$\Phi \left( \begin{bmatrix} 0 & x \\ \bar{x} & 0 \end{bmatrix} \right) = \begin{bmatrix} 0 & \lambda \\ \bar{\lambda} & 0 \end{bmatrix}$ , where $\lambda \in \C$ with $|\lambda|=1$.
\end{Lemma}

\begin{proof}
Take $a,c\in \RR$ with $a\neq c$. Then
$$ \Phi \left( \begin{bmatrix} 0 & x \\ \bar{x} & 0 \end{bmatrix} \begin{bmatrix} a & 0 \\ 0 & c \end{bmatrix} \begin{bmatrix} 0 & x \\ \bar{x} & 0 \end{bmatrix} \right) = \Phi \left( \begin{bmatrix} c & 0 \\ 0 & a \end{bmatrix} \right)= \beta(ac) \begin{bmatrix} c & 0 \\ 0 & a \end{bmatrix}.$$
Since $\begin{bmatrix} 0 & x \\ \bar{x} & 0 \end{bmatrix}$ is an involution for $|x|=1$, it must be that
$$\Phi \left( \begin{bmatrix} 0 & x \\ \bar{x} & 0 \end{bmatrix} \right) = \begin{bmatrix} \pm \sqrt{ 1- |\lambda|^2} & \lambda \\ \bar{\lambda} & \mp \sqrt{1-|\lambda|^2} \end{bmatrix}$$
for some $\lambda \in \C$ with $|\lambda|\leq 1$.
Then
$$\beta(ac) \begin{bmatrix} a+(c-a)|\lambda|^2 & \pm \lambda (a-c) \sqrt{1-|\lambda|^2} \\ \bar{\lambda} (a-c) \sqrt{1-|\lambda|^2} & c+(a-c)|\lambda|^2 \end{bmatrix} = \beta(ac) \begin{bmatrix} c & 0 \\ 0 & a \end{bmatrix},$$
from which it follows that $\lambda (a-c) \sqrt{1-|\lambda|^2} = 0$. If $\lambda = 0$, then $a=c$, which is a contradiction. Hence it must be that $\sqrt{1-|\lambda|^2} = 0$, thus $|\lambda|=1$.
\end{proof}

From the previous lemma it follows that there exists a function on a unit circle $\lambda:\Gamma \to \Gamma$ with $\lambda(1)=1$ such that $\Phi \left( \begin{bmatrix} 0 & x \\ \bar{x} & 0 \end{bmatrix} \right)= \begin{bmatrix} 0 & \lambda(x) \\ \overline{\lambda(x)} & 0 \end{bmatrix}$.

\begin{Lemma}
Assume the set of assumptions \textup{(C2)}, $\Phi(-I)=-I$, and $\gamma \equiv \id$. For arbitrary $x,y\in \Gamma$ we have
$\lambda(x y)= \lambda(x) \lambda(y)$ and $\lambda(\bar{x})=\overline{\lambda(x)}$.
\end{Lemma}

\begin{proof}
Take
$$\Phi \left( \begin{bmatrix} 0 & x \\ \bar{x} & 0 \end{bmatrix} \begin{bmatrix} 0 & y \\ \bar{y} & 0 \end{bmatrix} \begin{bmatrix} 0 & x \\ \bar{x} & 0 \end{bmatrix} \right) = \Phi \left( \begin{bmatrix} 0 & x^2 \bar{y} \\ \bar{x}^2 y & 0 \end{bmatrix} \right).$$
Hence $\lambda(x) \lambda(\bar{y}) \lambda(x)=\lambda(x^2 \bar{y})$. Taking $y=1$, we get $\lambda(x)^2=\lambda(x^2)$.

Next, take $x^2=z$ and $\bar{y}=u$. Then
$$\lambda(zu)=\lambda(x) \lambda(u) \lambda(x)=\lambda(z) \lambda(u)$$
and also $\lambda(1)=1$. Taking $|x|=1$, we get $\lambda(\bar{x})= \lambda(x^{-1})= \lambda(x)^{-1}= \overline{\lambda(x)}$, from which the second assertion follows.
\end{proof}

\begin{Lemma}
Under assumptions \textup{(C2)}, $\Phi(-I)=-I$ and $\gamma \equiv \id$ the function
$\lambda$ has one of the following forms:
$$\lambda(x)=x \quad \textrm{for every } x\in \Gamma \qquad \textrm{ or } \qquad \lambda(x)=\bar{x} \quad \textrm{for every } x\in \Gamma.$$
\end{Lemma}

\begin{proof}
For $x\in\Gamma$ write $x= e^{i \phi}$ for some $\phi \in [0,2\pi)$. Take
\begin{align*}
\Phi \left( \begin{bmatrix} a & x \\ \bar{x} & 0 \end{bmatrix} \right) &= \Phi \left( \begin{bmatrix} 0 & e^{\frac{i\phi}{2}} \\ e^{-\frac{i\phi}{2}} & 0 \end{bmatrix} \begin{bmatrix} 0 & 1 \\ 1 & a \end{bmatrix} \begin{bmatrix} 0 & e^{\frac{i\phi}{2}} \\ e^{-\frac{i\phi}{2}} & 0 \end{bmatrix} \right) \\
&= \begin{bmatrix} 0 & \lambda(e^{\frac{i\phi}{2}}) \\ \lambda(e^{-\frac{i\phi}{2}}) & 0 \end{bmatrix} \! \begin{bmatrix} 0 & 1 \\ 1 & a \end{bmatrix} \! \begin{bmatrix} 0 & \lambda(e^{\frac{i\phi}{2}}) \\ \lambda(e^{-\frac{i\phi}{2}}) & 0 \end{bmatrix} = \begin{bmatrix} a & \lambda(x) \\ \lambda(\bar{x}) & 0 \end{bmatrix}.
\end{align*}
Then
\begin{align*}
\Phi \left( \begin{bmatrix} x+\bar{x} & x^2 \\ \bar{x}^2 & 0 \end{bmatrix} \right) &= \Phi \left( \begin{bmatrix} 1 & x \\ \bar{x} & 0 \end{bmatrix} \begin{bmatrix} 0 & 1 \\ 1 & 0 \end{bmatrix} \begin{bmatrix} 1 & x \\ \bar{x} & 0 \end{bmatrix} \right) \\
&= \begin{bmatrix} 1 & \lambda(x) \\ \overline{\lambda(x)} & 0 \end{bmatrix} \begin{bmatrix} 0 & 1 \\ 1 & 0 \end{bmatrix} \begin{bmatrix} 1 & \lambda(x) \\ \overline{\lambda(x)} & 0 \end{bmatrix}\\ & = \begin{bmatrix} \lambda(x)+\lambda(\bar{x}) & \lambda(x^2) \\ \lambda(\bar{x}^2) & 0 \end{bmatrix} = \begin{bmatrix} x+ \bar{x} & \lambda(x^2) \\ \lambda(\bar{x}^2) & 0 \end{bmatrix}.
\end{align*}
Thus we have $\lambda(x)+\lambda(\bar{x})=x+\bar{x}$, which implies $\operatorname{Re}(\lambda(x))=\operatorname{Re}(x)$. Since $|\lambda(x)|=|x|=1$, we have either $\lambda(x)= x$ or $\lambda(x)= \bar{x}$. It is clear that these two forms of $\lambda$ cannot exist simultaneously, hence $\lambda$ always takes a single form for every $x \in \Gamma$.
\end{proof}

If $\lambda(x)=\bar{x}$, define $\Phi'(A)=\Phi(\bar{A})$. Therefore we can assume without the loss of generality that $\lambda(x)=x$.

\begin{Lemma}
Under assumptions \textup{(C2)}, $\Phi(-I)=-I$, $\gamma \equiv \id$, and $\lambda \equiv \id$ we have
$\Phi \left( \begin{bmatrix} a & b \\ \bar{b} & c \end{bmatrix} \right) = \beta(ac-b\bar{b}) \begin{bmatrix} a & b \\ \bar{b} & c \end{bmatrix}$.
\end{Lemma}

\begin{proof}
Take $a,c \in \RR$ and $b\in\C$. Write $b=|b| e^{i\phi}$ with $\phi \in [0,2\pi)$ and define $x=e^{i\frac{\phi}{2}}$.
Then 
\begin{align*}
\Phi \left( \begin{bmatrix} a & b \\ \bar{b} & c \end{bmatrix} \right) &= \Phi \left( \begin{bmatrix} 0 & x \\ \bar{x} & 0 \end{bmatrix} \begin{bmatrix} c & |b| \\ |b| & a \end{bmatrix} \begin{bmatrix} 0 & x \\ \bar{x} & 0 \end{bmatrix} \right) \\
& =  \begin{bmatrix} 0 & x \\ \bar{x} & 0 \end{bmatrix} \beta(ac-|b|^2) \begin{bmatrix} c & |b| \\ |b| & a \end{bmatrix} \begin{bmatrix}  0 & x \\ \bar{x} & 0 \end{bmatrix} = \beta(ac-b \bar{b}) \begin{bmatrix} a & b \\ \bar{b} & c \end{bmatrix},
\end{align*}
which concludes the proof.
\end{proof}

We have proved the following proposition.

\begin{Proposition} \label{thm6.23}
Let $\Phi: \s_2(\C) \to \s_2(\C)$ be a regular J.T.P. homomorphism such that
\begin{itemize}
	\item $\Phi(A)=0$ for every $A\in \s_2(\C)$ with $\rank A \leq 1$;
	\item $\Phi$ maps scalars to scalars;
	\item images of $\begin{bmatrix} 0 & 1 \\ 1 & 0 \end{bmatrix}$ and $\begin{bmatrix} -1 & 0 \\ 0 & 1 \end{bmatrix}$ don't commute.
\end{itemize}
Then there exist a unitary matrix $U$ and $\beta: \RR \to \RR$ a unital multiplicative map with $\beta(-1)=1$ such that
$$\Phi(A)= \begin{cases} \beta(\det A)\cdot U \widetilde{\Phi}(A) U^*; & \rank A=2 \\ 0 & \rank A \leq 1 \end{cases},$$
where $\widetilde{\Phi}$ has one of the following forms:
		\begin{multicols}{2}
		\begin{itemize}
			\item $\wt(A)=A$;
			\item $\wt(A)=\bar{A}$;
			\item $\wt(A)=A^{-1}$;
			\item $\wt(A)=\bar{A}^{-1}$;
			\item $\wt(A)=\eta(A)A$;
			\item $\wt(A)=\eta(A) \bar{A}$;
			\item $\wt(A)=\eta(A) A^{-1}$;
			\item $\wt(A)=\eta(A) \bar{A}^{-1}$;
		\end{itemize}
		\end{multicols}
for every $A\in \s_2(\C)$, where $\eta$ is the function defined in Lemma \ref{eta}.
\end{Proposition}

This case is covered by the form (vi) of Theorem \ref{main}.

\end{document}